\documentclass[preprint,3p]{elsarticle}
\usepackage{setspace}
\usepackage{amssymb,amsfonts,amsthm,amsmath,mathrsfs}
\usepackage{ifthen,bookmark}
\usepackage{epsfig}
\usepackage{lineno,hyperref}
\usepackage{booktabs}
\usepackage{multirow}
\usepackage{bm}
\usepackage{indentfirst}
\usepackage{array}
\usepackage{arydshln}
\usepackage{listings}
\usepackage{mathtools}
\usepackage{float}
\usepackage{verbatim}
\usepackage{extarrows} 
\usepackage[usenames]{color}
\usepackage[T1]{fontenc}
\usepackage{graphicx}
\usepackage{subfigure}
\usepackage{caption}
\usepackage{subeqnarray}
\usepackage{cases}
\usepackage{color}

\hypersetup{hidelinks}
\numberwithin{equation}{section}

\newtheorem{lemma}{Lemma}[section]
\newtheorem{theorem}{Theorem}[section]
\newtheorem{remark}{Remark}

\nonstopmode
\begin{document}


\begin{frontmatter}
\def\ll{\mbox{ } \hskip 1em}
\def\matrix{\,\vcenter\bgroup\plainLet@\plainvspace@
    \normalbaselines
   \math\ialign\bgroup\hfil$##$\hfil&&\quad\hfil$##$\hfil\crcr
      \mathstrut\crcr\noalign{\kern-\baselineskip}}
\def\mathbb{}

\title{Second-order accurate numerical scheme with graded meshes for the nonlinear partial integrodifferential equation arising from viscoelasticity}

\author[mymainaddress1]{Wenlin Qiu}  
\author[mymainaddress1]{Xu Xiao}
\author[mymainaddress1]{Kexin Li\corref{mycorrespondingauthor}}
\ead{likexin@hunnu.edu.cn}

\cortext[mycorrespondingauthor]{Corresponding author.}
\address[mymainaddress1]{Key Laboratory of Computing and Stochastic Mathematics (Ministry of Education), School of Mathematics and Statistics, Hunan Normal University, Changsha, Hunan 410081, P. R. China}

\begin{abstract}
 This paper establishes and analyzes a second-order accurate numerical scheme for the nonlinear partial integrodifferential equation with a weakly singular kernel. In the time direction, we apply the Crank-Nicolson method for the time derivative, and the product-integration (PI) rule is employed to deal with Riemann-Liouville fractional integral. From which, the non-uniform meshes are utilized to compensate for the singular behavior of the exact solution at $t=0$ so that our method can reach second-order convergence for time. In order to formulate a fully discrete implicit difference scheme, we employ a standard centered difference formula for the second-order spatial derivative, and the Galerkin method based on piecewise linear test functions is used to approximate the nonlinear convection term. Then we derive the existence and uniqueness of numerical solutions for the proposed implicit difference scheme. Meanwhile, stability and convergence are proved by means of the discrete energy method. Furthermore, to demonstrate the effectiveness of the proposed method, we utilize a fixed point iterative algorithm to calculate the discrete scheme. Finally, numerical experiments illustrate the feasibility and efficiency of the proposed scheme, in which numerical results are consistent with our theoretical analysis.
\end{abstract}


\end{frontmatter}



\section{Introduction}

\vskip 0.5mm
   In this work, we shall consider a nonlinear partial integrodifferential equation with a weakly singular kernel
   \begin{equation}\label{eq1.1}
   \begin{array}{ll}
     u_t(x,t) + u(x,t)u_x(x,t)   - \mathcal{I}^{(\alpha)}u_{xx}(x,t)  =  f(x,t), \qquad x\in (0,L), \qquad t \in (0,T],
   \end{array}
   \end{equation}
which subjects to the following boundary conditions
   \begin{equation}\label{eq1.2}
   \begin{array}{ll}
     u(0,t)=u(L,t)=0,  \qquad t \in (0,T],
   \end{array}
   \end{equation}
and the initial condition
   \begin{equation}\label{eq1.3}
    \begin{array}{ll}
    u(x,0)=u_0(x), \qquad x\in [0,L],
   \end{array}
   \end{equation}
where the Riemann-Liouville (R-L) fractional integral is denoted by \cite{Podlubny}
   \begin{equation}\label{eq1.4}
    \begin{array}{ll}
     \mathcal{I}^{(\alpha)}\varpi(t):= \int_0^t \beta(t-\zeta) \varpi(\zeta) d\zeta, \qquad  0 < \alpha < 1,
   \end{array}
  \end{equation}
from which, Abel kernel $\beta(t)=t^{\alpha-1}/\Gamma(\alpha)$ and $\Gamma(q)=\int_{0}^{+\infty}s^{q-1}\exp(-s)ds$ is the Euler's Gamma function.

\vskip 0.5mm
Sanz-Serna \cite{Sanz} pointed out that \eqref{eq1.1} affords a simple model which combines the Eulerian derivative, that is, $\partial u(x,t)/\partial t +u(x,t)\partial u(x,t)/\partial x$, with a viscoelastic effect, just like Burgers equation \cite{Burgers} provides a simple model for the studies of more realistic situations, involving Eulerian derivatives and viscous forces. In addition, problem \eqref{eq1.1}-\eqref{eq1.3} can model the physical phenomena, which involves the viscoelastic forces, population dynamics, viscous plasticity problems, heat transfer materials with memory, nuclear reaction theory and so on \cite{Olm,Podlubny,Sanz,Sheng}. Equation \eqref{eq1.1} has the significant applications in the fields of engineering and science, and it is still worthy of our further study. However, there are no analytic solutions for the problem \eqref{eq1.1}-\eqref{eq1.3}, thus we have to yield its numerical solutions.

\vskip 0.5mm
  In fact, for problem \eqref{eq1.1}-\eqref{eq1.3}, many studies of numerical fields have been considered in recent years, especially finite difference methods. First, Lopez-Marcos \cite{Lopez-Marcos} developed a backward Euler (BE) scheme for the time derivative, and the R-L integral term was approximated by the first-order convolution quadrature rule. Then, Tang \cite{Tang} considered this problem by utilizing a Crank-Nicolson (CN) method for the time derivative, and the R-L integral term was discretized by the trapezoidal product method. Further, Chen and Xu \cite{Chen1} proposed a formally second-order backward differentiation formula (BDF) scheme for the time derivative, and the R-L integral term was dealt with second-order convolution quadrature rule. After that, Zheng et al. \cite{Zheng} considered a CN-type method for the time derivative, and they used the trapezoidal convolution quadrature rule to approximate the R-L integral term. The theoretical analysis regarding the stability and convergence was reported in these articles. Unfortunately, these studies considered the temporal uniform mesh, which is still affected by the singular behavior of the exact solution at $t=0$, so that it is impossible to achieve accurate second-order convergence for time. Furthermore, in some researches above, numerical algorithms and simulations have been not given to illustrate the effectiveness of their numerical schemes, except for certain linear examples in \cite{Tang} and Crandall's finite difference scheme in \cite{Dehghan}. These facts all urge us to establish a temporal high-order scheme and its algorithm implementation to the numerical solutions of problem \eqref{eq1.1}-\eqref{eq1.3}.

\vskip 0.5mm
The main contributions of this work is presented as follows. (i) Based on the non-uniform meshes, we eliminate the singular behavior of the exact solution at $t=0$, and obtain the exact second-order implicit difference scheme of the non-linear problem \eqref{eq1.1}-\eqref{eq1.3}. (ii) For the constructed implicit difference scheme, with some suitable hypotheses, we prove the existence and uniqueness of the numerical solution, and derive the stability and convergence of the numerical scheme. (iii) In order to illustrate the effectiveness of the implicit difference scheme, we adopt an iterative algorithm to calculate and implement that. The numerical results show that the proposed scheme can reach the accurate second order in the space-time directions, which is consistent with our theoretical results.

\vskip 0.5mm
The rest of this paper is organized as follows. In Section 2, some preliminaries regarding temporal/spatial discretizations are given. Then Section 3 is devoted to the establishment of a fully discrete implicit difference scheme. In Section 4, we prove the existence and uniqueness of numerical solutions and derive the stability and convergence of the proposed scheme. With a fixed point iterative algorithm, numerical experiments are carried out to verify our theoretical results in Section 5. Ultimately, Section 6 gives the concluding remarks.

\section{Preliminaries}

  \subsection{Notations for temporal discretizations}

    \vskip 0.2mm
  Here, we present the temporal levels $0 = t_0 < t_1 < t_2 < \cdots$ and define the symbols
 \begin{equation*}
 \begin{array}{ll}
    k_n := t_n - t_{n-1},  \qquad  t_{n-\frac{1}{2}} := \frac{1}{2} (  t_{n} + t_{n-1} ), \qquad n\geq 1.
 \end{array}
  \end{equation*}

    \vskip 0.2mm
    Furthermore, define the grid function $W_k=\{W^n| 0\leq n\leq N\}$ and notations
 \begin{equation*}
 \begin{array}{ll}
    \delta_t W^{n-\frac{1}{2}} :=\frac{1}{k_n}(W^n-W^{n-1}), \qquad   W^{n-1/2}:=\frac{1}{2}(W^n+W^{n-1}),  \qquad  1\leq n \leq N,
 \end{array}
  \end{equation*}
 and the piecewise constant approximation
      \begin{equation}\label{eq2.1}
      \overline{W}(t):=
     \begin{cases}
          W^1, &   t_0 < t < t_1, \\
          W^{n-1/2}, &   t_{n-1} < t < t_n, \quad n\geq 2.
     \end{cases}
     \end{equation}

    \vskip 0.2mm
    Then, employing $\overline{W}$, we denote the following discrete fractional integral (see \cite{McLean})
     \begin{equation}\label{eq2.2}
     \begin{split}
           I^{(\alpha)}W^{n-1/2} &= \frac{1}{k_n}\int_{t_{n-1}}^{t_n}\mathcal{I}^{(\alpha)}\overline{W}(t)dt = \frac{1}{k_n}\int_{t_{n-1}}^{t_n}\int_{t_0}^{t}\beta(t-\zeta)\overline{W}(\zeta)d\zeta dt \\
           & = \widehat{w}_{n1}W^1k_1 + \sum\limits_{s=2}^{n}\widehat{w}_{ns} W^{s-1/2}k_s,
     \end{split}
     \end{equation}
where we utilize the PI rule, besides,
     \begin{equation}\label{eq2.3}
     \begin{split}
           \widehat{w}_{ns} = \frac{1}{k_nk_s}\int_{t_{n-1}}^{t_n}\int_{t_{s-1}}^{\min\{t,t_s\}}\beta(t-\zeta)  d\zeta dt >0.
     \end{split}
     \end{equation}

    \vskip 0.2mm
Then for $n\geq 2$, it holds that
     \begin{equation}\label{eq2.4}
     \begin{split}
           \widehat{w}_{ns} &= \frac{1}{k_nk_s}\int_{t_{n-1}}^{t_n}\int_{t_{s-1}}^{t_s}\beta(t-\zeta)  d\zeta dt \\
           & = \frac{ \left[ (t_{n}-t_{s-1})^{\alpha+1} - (t_{n}-t_{s})^{\alpha+1} \right] - \left[ (t_{n-1}-t_{s-1})^{\alpha+1} - (t_{n-1}-t_{s})^{\alpha+1} \right] }{k_n k_s \Gamma(\alpha+2)}
     \end{split}
     \end{equation}
and
     \begin{equation}\label{eq2.5}
     \begin{split}
           \widehat{w}_{nn} = \frac{1}{k_n^2}\int_{t_{n-1}}^{t_n}\int_{t_{n-1}}^{t}\beta(t-\zeta)  d\zeta dt = \frac{k_n^{\alpha-1}}{\Gamma(\alpha+2)}, \qquad n\geq 1.
     \end{split}
     \end{equation}

    \vskip 0.2mm
    Additionally, regarding the source term $f$, we give the approximation as follows
     \begin{equation}\label{eq2.6}
     \begin{split}
           f^{n-1/2} \approx \frac{1}{k_n}\int_{t_{n-1}}^{t_n} f(\cdot,t) dt,  \qquad n\geq 1,
     \end{split}
     \end{equation}
assumed to meet that
     \begin{equation}\label{eq2.7}
     \begin{split}
            \left\|  f^{\frac{1}{2}}k_1 - \int_{t_0}^{t_1} f(\cdot,t)dt  \right\| \leq C \int_{t_0}^{t_1} t \|f'(\cdot,t)\| dt,
     \end{split}
     \end{equation}
and for $n\geq 2$,
     \begin{equation}\label{eq2.8}
     \begin{split}
            \left\|  f^{n-1/2}k_n - \int_{t_{n-1}}^{t_n} f(\cdot,t)dt  \right\| \leq C k_n^2 \int_{t_{n-1}}^{t_n}  \|f''(\cdot,t)\| dt.
     \end{split}
     \end{equation}
Here, for example, $f^{n-\frac{1}{2}}=f(\cdot,t_{n-1/2})$ or $f^{n-\frac{1}{2}}= \frac{1}{2}[ f(\cdot,t_{n-1}) + f(\cdot,t_{n}) ]$ is also permissible.

    \vskip 0.5mm
   Next, for eliminating the singular behaviour of the exact solution at $t=0$ and obtaining accurate second order for time, the following hypotheses regarding non-uniform meshes \cite{McLean} is presented by
     \begin{equation}\label{eq2.9}
     \begin{split}
            k_n\leq \widetilde{C}_{\gamma}k \min\left\{1,t_n^{1-\frac{1}{\gamma}}\right\}, \qquad n\geq 1, \qquad \gamma\geq 1,
     \end{split}
     \end{equation}
from which, $\widetilde{C}_{\gamma}$ is independent of $k$,
      \begin{equation}\label{eq2.10}
     \begin{split}
                 t_1=k_1\geq \widetilde{c}_{\gamma}k^{\gamma}, \qquad  t_n \leq \widetilde{C}_{\gamma}t_{n-1},    \qquad n\geq 2,
     \end{split}
     \end{equation}
and a more rigid assumption about the temporal mesh, i.e.,
      \begin{equation}\label{eq2.11}
     \begin{split}
             0\leq k_{n+1} - k_n\leq \widetilde{C}_{\gamma}k^2 \min\left\{1,t_n^{1-2/\gamma}\right\}, \qquad n\geq 2.
     \end{split}
     \end{equation}

   \vskip 0.2mm
   Thus for $t_0\leq t\leq T$, the case satisfying above three hypotheses \eqref{eq2.9}-\eqref{eq2.11} is
      \begin{equation}\label{eq2.12}
     \begin{split}
          t_n = (nk)^{\gamma}, \qquad 0\leq n \leq N, \qquad k= \frac{T^{1/\gamma}}{N}.
     \end{split}
     \end{equation}

   \subsection{Notations of spatial discretizations}
First, denote nodes $x_j=jh\; (0\leq j\leq J)$ with $h=\frac{L}{J}$ for the positive integer $J$. Let $\Omega_h=\{x_j | 0\leq j\leq J \}$. We define the grid function $\mathcal{W}=\{ W_{j}|\,\, 0\leq j \leq J\}$. Then we give the following notations
  \begin{equation*}
  \begin{array}{ll}
     \Delta_{x}W_j=\frac{1}{2h}(W_{j+1}-W_{j-1}), \quad \delta_{x}W_{j-\frac{1}{2}}=\frac{1}{h}(W_j-W_{j-1}), \quad T_+W_j=W_{j+1}, \quad T_-W_j=W_{j-1}, \\
     \Delta W_j=T_+W_j - T_-W_j,  \qquad   \Delta_{+}W_j=W_{j+1} - W_{j}, \qquad \Delta_{-} W_j = W_{j} - W_{j-1}, \\
     \delta_{x}^{2}W_j=\frac{1}{h}(\delta_{x}W_{j+\frac{1}{2}}-\delta_{x}W_{j-\frac{1}{2}}), \qquad \nabla W_j=\frac{1}{3}(W_{j-1}+W_{j}+u_{j+1}).
 \end{array}
 \end{equation*}

    \vskip 0.2mm
Then we present the following lemma.
  \begin{lemma}\label{lemma2.1}
  \cite{Zhang}. Assuming $G(x)\in C_{x}^{4}([0,L])$, we can get
  $$\left.
  \begin{array}{ll}
   G''(x_{j})=\delta_{x}^{2}G_{j}-\frac{h^{2}}{6}\int_{0}^{1} \Big( G''''(x_{j}+\zeta h)+G''''(x_{j}-\zeta h) \Big)(1-\zeta)^{3}d\zeta.
  \end{array}
 \right.
 $$
 \end{lemma}

 \vskip 0.2mm
   \begin{remark} \label{remark1}
    Throughout this paper, $C$ denotes a general positive numeber, which may be various in different situations, however, independent of the space-time step sizes $k_n$ with $1\leq n \leq N$ and $h$.
 \end{remark}

\section{Construction of second-order implicit difference scheme}
 \vskip 0.2mm
 Below, we can establish a fully discrete second-order implicit difference scheme for the problem \eqref{eq1.1}-\eqref{eq1.3}.

 \vskip 0.2mm
 Firstly, denote the following grid functions
    \begin{equation*}
    \begin{array}{ll}
         u^n = u(x,t_n), \qquad  0\leq n\leq N,  \qquad   f^{n} =f(x,t_{n}),  \qquad  1\leq n\leq N,
    \end{array}
    \end{equation*}
    \begin{equation*}
    \begin{array}{ll}
         u_j^n = u(x_j,t_n), \quad  0\leq n\leq N,  \quad   f_j^{n} = f(x_j,t_{n}),  \quad  1\leq n\leq N,  \quad 0\leq j \leq J,
    \end{array}
    \end{equation*}
and assume that
    \begin{equation*}
    \begin{array}{ll}
     c_0 := \max \limits_{   (x,t)\in [0,L]\times(0,T] } \left\{ |u(x,t)|, \left|\frac{\partial}{\partial x}u(x,t)\right| \right\}.
    \end{array}
    \end{equation*}

  \vskip 0.2mm
 Then we integrate \eqref{eq1.1} from $t=t_{n-1}$ to $t=t_{n}$ and multiply $\frac{1}{k_n}$ to obtain
 \begin{equation}\label{eq3.1}
 \begin{split}
   \delta_tu^{n-\frac{1}{2}} & + \frac{1}{k_n} \int_{t_{n-1}}^{t_n} uu_x dt = \frac{1}{k_n} \int_{t_{n-1}}^{t_n}\int_{t_0}^{t}\beta(t-\zeta)u_{xx}(x,\zeta)d\zeta dt \\
   & + \frac{1}{k_n} \int_{t_{n-1}}^{t_n}f(x,t)dt, \qquad x\in (0,L), \qquad 1\leq n\leq N,
 \end{split}
  \end{equation}
 from which, we have
 \begin{equation}\label{eq3.2}
 \begin{split}
   \delta_tu^{n-\frac{1}{2}} & + \frac{1}{k_n} \int_{t_{n-1}}^{t_n} uu_x dt = \frac{1}{k_n} \int_{t_{n-1}}^{t_n}\int_{t_0}^{t}\beta(t-\zeta) \overline{u}_{xx}(x,\zeta)d\zeta dt \\
   & + f^{n-1/2} + (R_1)^{n-1/2},  \qquad x\in (0,L), \qquad 1\leq n\leq N,
 \end{split}
  \end{equation}
 where $(R_1)^{n-\frac{1}{2}} =  (R_{11})^{n-\frac{1}{2}} +  (R_{12})^{n-\frac{1}{2}}$ with
 \begin{equation}\label{eq3.3}
 \begin{split}
   (R_{11})^{n-1/2} =  \frac{1}{k_n} \int_{t_{n-1}}^{t_n}f(x,t)dt- f^{n-1/2},
 \end{split}
  \end{equation}
 and
 \begin{equation}\label{eq3.4}
 \begin{split}
   (R_{12})^{n-1/2} =\frac{1}{k_n} \int_{t_{n-1}}^{t_n}\int_{t_0}^{t}\beta(t-s)\Big[u_{xx}(x,s)-\overline{u}_{xx}(x,s)\Big]ds dt,
 \end{split}
  \end{equation}
  where
   \begin{equation}\label{eq3.5}
 \begin{split}
    f^{n-1/2} = \frac{1}{2} \Big(  f(x,t_{n}) + f(x,t_{n-1})  \Big).
 \end{split}
  \end{equation}

    \vskip 0.2mm
    Then for second term in \eqref{eq3.2}, we employ the right rectangle formula and the middle rectangle formula to get respectively
 \begin{equation}\label{eq3.6}
 \begin{split}
         \frac{1}{k_1} \int_{t_{0}}^{t_1} uu_x dt = u(x,t_1)u_x(x,t_1) + \mathcal{O}(k_1),
 \end{split}
  \end{equation}
 \begin{equation}\label{eq3.7}
 \begin{split}
         \frac{1}{k_n} \int_{t_{n-1}}^{t_n} uu_x dt = u(x,t_{n-1/2})u_x(x,t_{n-1/2}) + \mathcal{O}(k_n^2),  \qquad  2\leq n \leq N.
 \end{split}
  \end{equation}

Next, we employ \eqref{eq3.2}, \eqref{eq3.6} and \eqref{eq3.7} to obtain that
 \begin{equation}\label{eq3.8}
 \begin{split}
   \delta_tu^{\frac{1}{2}} & + u(x,t_1)u_x(x,t_1) = I^{(\alpha)} u_{xx}^{\frac{1}{2}}
    + f^{\frac{1}{2}} + (R_1)^{\frac{1}{2}} + \mathcal{O}(k_1),  \quad x\in (0,L),
 \end{split}
  \end{equation}
and
 \begin{equation}\label{eq3.9}
 \begin{split}
   \delta_tu^{n-\frac{1}{2}} & + u(x,t_{n-\frac{1}{2}})u_x(x,t_{n-\frac{1}{2}})  = I^{(\alpha)}u_{xx}^{n-1/2}   + f^{n-1/2} \\
   & + (R_1)^{n-1/2}  + \mathcal{O}(k_n^2), \qquad x\in (0,L),  \qquad 2\leq n\leq N.
 \end{split}
  \end{equation}


    \vskip 2mm
    Then, we construct the fully discrete implicit difference scheme based on the spacial difference approximation.

    \vskip 0.2mm
   For \eqref{eq3.8} and \eqref{eq3.9}, we consider them at the point $x_j$, and use Lemma \ref{lemma2.1} to get
 \begin{equation}\label{eq3.10}
 \begin{split}
   \delta_tu_j^{\frac{1}{2}}  + u(x_j,t_1)u_x(x_j,t_1) &= I^{(\alpha)} \delta_x^2u_j^{\frac{1}{2}}
    + f_j^{\frac{1}{2}} + (R_1)_j^{\frac{1}{2}}+  \mathcal{O}\left( \left( \frac{t_1^{\alpha}-t_0^{\alpha}}{k_1\Gamma(\alpha+1)}\right)   h^2  \right) \\
    & + \mathcal{O}(k_1+h^2),  \qquad 1\leq j \leq J-1,
 \end{split}
  \end{equation}
and
 \begin{equation}\label{eq3.11}
 \begin{split}
   \delta_tu_j^{n-\frac{1}{2}} & + u(x_j,t_{n-\frac{1}{2}})u_x(x_j,t_{n-\frac{1}{2}})  = I^{(\alpha)}\delta_x^2u_j^{n-1/2}   + f_j^{n-1/2} + (R_1)_j^{n-1/2} \\
   & + \mathcal{O}\left( \left( \frac{t_n^{\alpha}-t_{n-1}^{\alpha}}{k_n\Gamma(\alpha+1)}\right)   h^2  \right)  + \mathcal{O}(k_n^2 + h^2),   \qquad 1\leq j \leq J-1,   \qquad 2\leq n\leq N,
 \end{split}
  \end{equation}
from which, we utilize the Galerkin method based on piecewise linear test functions \cite{Lopez-Marcos} to approximate the nonlinear convection term
 \begin{equation}\label{eq3.12}
 \begin{split}
     u(x_j,t_1)u_x(x_j,t_1) =  \nabla u^{1}_j \Delta_x  u^{1}_j + \mathcal{O}(h^2),
 \end{split}
  \end{equation}
 \begin{equation}\label{eq3.13}
 \begin{split}
     u(x_j,t_{n-\frac{1}{2}})u_x(x_j,t_{n-\frac{1}{2}}) =  \nabla u^{n-1/2}_j \Delta_x  u^{n-1/2}_j + \mathcal{O}(h^2),  \qquad 2\leq n\leq N.
 \end{split}
  \end{equation}

\vskip 2mm
Thus, we can get the following fully discrete implicit difference equations
 \begin{equation}\label{eq3.14}
 \begin{split}
   \delta_tu_j^{\frac{1}{2}} & +  \nabla u^{1}_j \Delta_x  u^{1}_j = I^{(\alpha)} \delta_x^2u_j^{\frac{1}{2}}
    + f_j^{\frac{1}{2}} + (R_1)_j^{\frac{1}{2}} + (R_2)_j^{\frac{1}{2}},  \quad 1\leq j \leq J-1,
 \end{split}
  \end{equation}
and
 \begin{equation}\label{eq3.15}
 \begin{split}
   \delta_tu_j^{n-\frac{1}{2}} & + \nabla u^{n-1/2}_j \Delta_x  u^{n-1/2}_j  = I^{(\alpha)}\delta_x^2u_j^{n-1/2}   + f_j^{n-1/2} \\
   & + (R_1)_j^{n-1/2}  + (R_2)_j^{n-1/2} ,   \qquad 1\leq j \leq J-1,   \qquad 2\leq n\leq N,
 \end{split}
  \end{equation}
where
 \begin{equation}\label{eq3.16}
 \begin{split}
      &\left|(R_2)_j^{\frac{1}{2}} \right| \leq C\left(k_1+h^2+ \left( \frac{t_1^{\alpha}-t_{0}^{\alpha}}{k_1\Gamma(\alpha+1)}\right)   h^2\right),  \\
        & \left| (R_2)_j^{n-\frac{1}{2}} \right| \leq  C\left(k_n^2+h^2+ \left( \frac{t_n^{\alpha}-t_{n-1}^{\alpha}}{k_n\Gamma(\alpha+1)}\right)   h^2\right), \qquad 2\leq n\leq N,
 \end{split}
  \end{equation}
which subject to the following initial and boundary conditions
 \begin{equation}\label{eq3.17}
 \begin{split}
     u^0_{j} = u_0(x_j),   \qquad 1\leq j \leq J-1,
 \end{split}
  \end{equation}
 \begin{equation}\label{eq3.18}
 \begin{split}
    u_0^n=u_J^n=0, \qquad 1\leq n\leq N.
 \end{split}
  \end{equation}

  \vskip 0.2mm
  Now, omitting $(R_1)_j^{n-\frac{1}{2}}$ and $(R_2)_j^{n-\frac{1}{2}}$ in \eqref{eq3.14}-\eqref{eq3.15} and replacing $u^n_{j}$ with its numerical approximation $U^n_{j}$, the fully discrete implicit difference scheme can be yielded by
 \begin{equation}\label{eq3.19}
 \begin{split}
   \delta_tU_j^{\frac{1}{2}} & +  \nabla U^{1}_j \Delta_x  U^{1}_j = I^{(\alpha)} \delta_x^2U_j^{1/2}
    + f_j^{1/2},  \qquad 1\leq j \leq J-1,
 \end{split}
  \end{equation}
 \begin{equation}\label{eq3.20}
 \begin{split}
   \delta_tU_j^{n-\frac{1}{2}} & + \nabla U^{n-\frac{1}{2}}_j \Delta_x  U^{n-\frac{1}{2}}_j  = I^{(\alpha)}\delta_x^2U_j^{n-\frac{1}{2}}   + f_j^{n-\frac{1}{2}}, \quad 1\leq j \leq J-1,   \quad 2\leq n\leq N,
 \end{split}
  \end{equation}
 \begin{equation}\label{eq3.21}
 \begin{split}
     U^0_{j} = u_0(x_j),   \qquad 1\leq j \leq J-1,
 \end{split}
  \end{equation}
 \begin{equation}\label{eq3.22}
 \begin{split}
    U_0^n=U_J^n=0, \qquad 1\leq n\leq N.
 \end{split}
  \end{equation}

\section{Theoretical analysis}
  \vskip 0.2mm
      In this section, we give some notations and lemmas for our theoretical results, including the existence and uniqueness of numerical solutions, and the stability and convergence of implicit difference scheme.

 \vskip 0.2mm
     Let $\mathbf{W}_h=\{w|w=(w_0, w_1, \cdots, w_J), w_0=w_J=0\}$. For any $w,v\in\mathbf{W}_h$, the discrete inner product and norms can be denoted via
     \begin{equation}\label{eq4.1}
     \begin{array}{ll}
     \langle w, v \rangle:=h\sum\limits_{s=1}^{J-1}w_sv_s,  \qquad
     \|w\|:=\sqrt{\langle w,w \rangle},   \qquad  \|w\|_{\infty}:=\max \limits_{0\leq s \leq J}\{ |w_s| \}.
     \end{array}
     \end{equation}

   \vskip 0.2mm
     \begin{lemma}\label{lemma4.1}\cite{Lopez-Marcos} For any $w, v\in \mathbf{W}_h$, it holds that
     	\begin{equation*}
     	\langle \delta_x^2w, v \rangle=-h\sum\limits_{s=0}^{J-1}{\delta_xw_{s+1}}{\delta_xv_{s+1}}.
     	\end{equation*}
     \end{lemma}

    \vskip 0.2mm
  \begin{lemma}\label{lemma4.2}\cite{Lopez-Marcos} For any $w,v \in \mathbf{W}_h$, we have
  \begin{equation*}
  \begin{array}{ll}
  \langle\Delta(wv), v\rangle=\frac{1}{2}\langle T_+v\Delta_+w+T_-v\Delta_-w, v\rangle.
  \end{array}
  \end{equation*}
 \end{lemma}

  \vskip 0.2mm
  \begin{lemma}\label{lemma4.3} For any $w,v \in \mathbf{W}_h$, we can yield
  \begin{equation*}
  \begin{array}{ll}
   (i)\, \langle w\Delta v+ \Delta(wv), v\rangle=0; \qquad
   (ii)\, \langle w\Delta w ,  v\rangle + \langle \Delta (wv) ,  w\rangle = 0; \\
   (iii)\, \langle w\Delta v ,  w\rangle + \langle \Delta(w)^2 ,  v\rangle = 0.
  \end{array}
  \end{equation*}
 \end{lemma}

  \vskip 0.2mm
 \begin{proof} $(i)$ From the definition of $\langle \cdot , \cdot\rangle$, we first yield that
  \begin{equation*}
  \begin{array}{ll}
  \langle w\Delta v+ \Delta(wv), v\rangle \\
  = h \sum\limits_{s=1}^{J-1}w_s(v_{s+1}-v_{s-1})v_s + h \sum\limits_{s=1}^{J-1}(w_{s+1}v_{s+1}-w_{s-1}v_{s-1})v_s \\
  = h \sum\limits_{s=1}^{J-1}w_sv_{s+1}v_s - h \sum\limits_{s=1}^{J-1}w_{s}v_{s-1}v_{s}
  + h \sum\limits_{s=1}^{J-1}w_{s+1}v_{s}v_{s+1} - h \sum\limits_{s=1}^{J-1}w_{s-1}v_{s-1}v_s  \\
  = h \sum\limits_{s=2}^{J}w_{s-1}v_{s-1}v_{s} - h \sum\limits_{s=1}^{J-1}w_{s-1}v_{s-1}v_{s}
  + h \sum\limits_{s=2}^{J}w_{s}v_{s-1}v_s - h \sum\limits_{s=1}^{J-1}w_{s}v_{s-1}v_s  \\
  = hw_{J-1}v_{J}v_{J-1} - hw_{0}v_{1}v_0 + hw_{J}v_{J-1}v_J - hw_{1}v_{0}v_1 \\
  = 0.
  \end{array}
  \end{equation*}
The proofs of $(ii)$ and $(iii)$ can be obtained analogously. Then we complete the proof.
 \end{proof}

 \vskip 0.2mm
  \begin{lemma}\label{lemma4.4} For any $w, v \in \mathbf{W}_h$, it holds that
  \begin{equation*}
  \begin{array}{ll}
  \big\langle v\Delta(v-w)+(v-w)\Delta w + \Delta(v-w)(v+w)  , v-w \big\rangle  \\
  = \big\langle (v-w)\Delta w + \Delta(w(v-w)) , v-w \big\rangle.
  \end{array}
  \end{equation*}
 \end{lemma}

  \vskip 0.2mm
 \begin{proof} First, using Lemma \ref{lemma4.3}, we get
   \begin{equation*}
  \begin{array}{ll}
  \big\langle v\Delta(v-w)+(v-w)\Delta w + \Delta(v-w)(v+w)  , v-w \big\rangle  \\
  = \big\langle v\Delta(v-w)+(v-w)\Delta w + \Delta(v-w)(v+w)  , v-w \big\rangle \\
  - \langle (v-w)\Delta (v-w)+\Delta(v-w)^2, v-w\rangle \\
  = \langle (v-w)\Delta w + w\Delta(v-w) + 2\Delta(w(v-w))  , v-w \rangle \\
  = \big\langle (v-w)\Delta w + \Delta(w(v-w))  , v-w \big\rangle
  + \big\langle w\Delta(v-w) + \Delta(w(v-w))  , v-w \big\rangle  \\
  = \mathcal{A}_1 + \mathcal{A}_2.
  \end{array}
  \end{equation*}
Then we only need to derive $\mathcal{A}_2=0$. Employing Lemma \ref{lemma4.3} again, we yield
  \begin{equation*}
  \begin{array}{ll}
     \mathcal{A}_2 = \langle w\Delta(v-w) + \Delta(w(v-w))  , v-w \rangle \\
      = \langle w\Delta(v-w) + \Delta(w(v-w))  , v \rangle
      - \langle w\Delta(v-w) + \Delta(w(v-w))  , w \rangle \\
      = \langle  w\Delta v - w\Delta w + \Delta(wv) - \Delta(w)^2  , v \rangle
      - \langle  w\Delta v - w\Delta w + \Delta(wv) - \Delta(w)^2  , w \rangle  \\
      = \langle  w\Delta v - w\Delta w + \Delta(wv) - \Delta(w)^2  , v \rangle
      - \langle  w\Delta v  + \Delta(wv)  , w \rangle  \\
      = \langle w\Delta v+ \Delta(wv), v\rangle  -  (\langle w\Delta w ,  v\rangle + \langle \Delta (wv) ,  w\rangle) - ( \langle w\Delta v ,  w\rangle + \langle \Delta(w)^2 ,  v\rangle )  \\
      = 0,
  \end{array}
  \end{equation*}
which finishes the proof.
 \end{proof}


  \subsection{Existence of numerical solutions}

  \vskip 0.2mm
  Herein, we use the Leray-Schauder theorem \cite{Ortega} to derive the existence of numerical solutions for the second-order implicit difference scheme \eqref{eq3.19}-\eqref{eq3.22}.

  \vskip 0.2mm
  \begin{theorem}\label{th4.1}
	Given $J, N\in \mathbb{Z}^+$ and $U^{0}\in\mathbb{R}^{J-1}$, then second-order implicit difference equations \eqref{eq3.19}-\eqref{eq3.20} have the solution $U^{n}$ with $1 \leq n \leq N$.
  \end{theorem}
	\begin{proof} With $U^{0}\in\mathbb{R}^{J-1}$, we first utilize \cite[Section 3]{Lopez-Marcos} to show that \eqref{eq3.19} has a solution $U^1$. Then we need to apply the mathematical induction to show that, provided $U^{s}$ for $2 \leq s \leq n-1$, the equation \eqref{eq3.20} for $U^n$ has a solution.

  \vskip 0.2mm
		Below define the mapping $\mathscr{Q}$: $\mathbb{R}^{J-1}\rightarrow\mathbb{R}^{J-1}$ by
		\begin{equation*}
		    \mathscr{Q}(v):=-\frac{k_n}{12h}(v\Delta v + \Delta(v)^{2}) + \frac{1}{2}k^2_n\widehat{w}_{nn}\delta_{x}^{2}v,
		\end{equation*}
		then $U^{n}$ is a solution of equation \eqref{eq3.20} if and only if
		\begin{equation*}
           \begin{array}{ll}
		       U^{n-1/2}=\mathscr{Q}(U^{n-1/2}) + \mathcal{W},
		 \end{array}
		\end{equation*}
		where
		\begin{equation*}
		 \mathcal{W} = U^{n-1} + \frac{1}{2}k_nk_1\widehat{w}_{n1}\delta_{x}^{2}U^1 + \frac{1}{2}k_n\sum\limits_{s=2}^{n-1}k_s\widehat{w}_{ns}\delta_{x}^{2}U^{s-1/2} + \frac{1}{2}k_n f^{n-1/2}.
		\end{equation*}

  \vskip 0.2mm
		Therefore, we have to illustrate that the mapping $\mathcal{P}(\cdot)=\mathscr{Q}(\cdot)+\mathcal{W}$ has a fixed point. Next, we consider an open ball
       $\mathscr{D}=\mathcal{B}(0,r)$ in $\mathbb{R}^{J-1}$ endowed with the norm $\|\cdot\|$ in \eqref{eq4.1}. Assume that for $\lambda>1$ and $U^{n-1/2}$ in the boundary of $\mathscr{D}$,
		\begin{equation}\label{eq4.2}
          \begin{array}{ll}
		       \lambda U^{n-1/2}=\mathcal{P}(U^{n-1/2})=\mathscr{Q}(U^{n-1/2}) + \mathcal{W}.
         \end{array}
		\end{equation}

  \vskip 0.2mm
        Then using Lemmas \ref{lemma4.1}-\ref{lemma4.3}, we have
        \begin{equation*}
		    \big\langle \mathscr{Q}(U^{n-1/2}),  U^{n-1/2}  \big\rangle \leq  0.
		\end{equation*}

  \vskip 0.2mm
		Now, we take the inner product of \eqref{eq4.2} with $U^{n-1/2}$ and apply the Cauchy-Schwarz inequality, then
		\begin{equation*}
		   \lambda \|U^{n-1/2}\|^2 \leq \langle \mathcal{W}, U^{n-1/2} \rangle \leq \|\mathcal{W}\|  \|U^{n-1/2}\|  \leq  \frac{1}{2}(\|\mathcal{W}\|^2 + \|U^{n-1/2}\|^2),
		\end{equation*}
	thus,
		\begin{equation*}
		    \lambda \leq  \frac{1}{2}\left( \frac{\|\mathcal{W}\|^2}{\|U^{n-1/2}\|^2} + 1 \right)  \leq  \frac{\|\mathcal{W}\|^{2}}{2r^{2}}  + \frac{1}{2}.
		\end{equation*}

  \vskip 0.2mm	
		As $r$ large, the above inequality contradicts with the assumption $\lambda>1$. Hence, \eqref{eq4.2} has no solution on
        $\partial\mathscr{D}$. Then employing the Leray-Schauder theorem (cf.~\cite{Ortega}, Thm.~6.3.3), we yield the existence of a fixed point of the mapping $\mathcal{P}$ in the closure of $\mathscr{D}$.
       \end{proof}

     \subsection{Stability analysis}
  \vskip 0.2mm
     We below use the energy method to prove the stability of the second-order implicit difference scheme \eqref{eq3.19}-\eqref{eq3.22}.

  \vskip 0.2mm
     \begin{theorem}\label{th4.2}
     	Supposing that $\{U_j^n \,|\, 1\leq j\leq J-1,\,\, 1\leq n\leq N\}$ is the solution of the second-order implicit difference scheme \eqref{eq3.19}-\eqref{eq3.22}, for $1\leq n\leq N$, we have
     \begin{equation*}
     \begin{array}{ll}
         \|U^n\| \leq  \|U^0\| + 2 \sum\limits_{l=1}^n k_l \|f^{l-1/2}\|.
     \end{array}
     \end{equation*}
     \end{theorem}

  \vskip 0.2mm
     \begin{proof} Taking the inner product of \eqref{eq3.19}-\eqref{eq3.20} with $U^{1}$ and $U^{n-1/2}$, respectively, we get
 \begin{equation}\label{eq4.3}
 \begin{split}
   \langle \delta_tU^{\frac{1}{2}} ,  U^{1}\rangle  & +  \langle\nabla U^{1} \Delta_x  U^{1} ,  U^{1}\rangle  = \langle I^{(\alpha)} \delta_x^2U^{1/2} ,  U^{1}\rangle
    + \langle f^{1/2} ,  U^{1}\rangle,
 \end{split}
  \end{equation}
  and
 \begin{equation}\label{eq4.4}
 \begin{split}
   \langle \delta_tU^{n-\frac{1}{2}} , U^{n-1/2} \rangle & + \langle \nabla U^{n-\frac{1}{2}} \Delta_x  U^{n-\frac{1}{2}}, U^{n-1/2} \rangle  = \langle  I^{(\alpha)}\delta_x^2U^{n-\frac{1}{2}}, U^{n-1/2} \rangle   \\
   & + \langle f^{n-\frac{1}{2}}, U^{n-1/2} \rangle, \qquad  2\leq n\leq N,
 \end{split}
  \end{equation}

  \vskip 0.2mm
    Then, employing Lemmas \ref{lemma4.3} and \ref{lemma4.1}, and for $N\geq 1$, we yield
     \begin{equation}\label{eq4.5}
      \begin{split}
       k_1\langle \delta_tU^{\frac{1}{2}} ,  U^{1}\rangle & +  \sum\limits_{n=2}^Nk_n \langle \delta_tU^{n-\frac{1}{2}}, U^{n-\frac{1}{2}} \rangle
        = k_1\langle f^{\frac{1}{2}} ,  U^{1}\rangle + \sum\limits_{n=2}^N k_n \langle f^{n-1/2}, U^{n-1/2}\rangle \\
        & - \left(  k_1\langle  I^{(\alpha)}\delta_xU^{\frac{1}{2}}, \delta_xU^{1} \rangle +  \sum\limits_{n=2}^N k_n \langle  I^{(\alpha)}\delta_xU^{n-\frac{1}{2}}, \delta_xU^{n-1/2} \rangle \right).
     \end{split}
    \end{equation}

  \vskip 0.2mm     	
     	Each term in \eqref{eq4.5} will be bounded one by one. At first, utilizing
     	\begin{equation*}
      \begin{split}	
     	\langle\delta_tU^{n-1/2}, U^{n}\rangle & =\frac{1}{2k_n}\langle U^{n}-U^{n-1}, U^{n}-U^{n-1}+U^{n}+U^{n-1} \rangle  \\
        &\geq \frac{1}{2k_n}(||U^{n}||^2-||U^{n-1}||^2),
        \end{split}
     	\end{equation*}
     	we have
     	\begin{equation}\label{eq4.6}
        \begin{array}{ll}
     	          k_1\langle \delta_tU^{\frac{1}{2}} ,  U^{1}\rangle  +  \sum\limits_{n=2}^N k_n \langle \delta_tU^{n-\frac{1}{2}}, U^{n-\frac{1}{2}} \rangle
                \geq  \frac{||U^{N}||^2-||U^{0}||^2}{2}.
         \end{array}
     	\end{equation}

   \vskip 0.2mm    	
     	Secondly, from \cite[pp.~483-485, (1.14)]{McLean}, we can obtain
     \begin{equation}\label{eq4.7}
      \begin{split}
         \left(  k_1\langle  I^{(\alpha)}\delta_xU^{\frac{1}{2}}, \delta_xU^{1} \rangle +  \sum\limits_{n=2}^N k_n \langle  I^{(\alpha)}\delta_xU^{n-\frac{1}{2}}, \delta_xU^{n-1/2} \rangle \right) \geq 0.
     \end{split}
    \end{equation}

   \vskip 0.2mm    	
     	Then, by substituting \eqref{eq4.6} and \eqref{eq4.7} into \eqref{eq4.5}, and utilizing the Cauchy-Schwarz inequality, we have
     	\begin{equation*}
        \begin{array}{ll}
     	\|U^N\|^2  \leq  \|U^0\|^2 + 2k_1\| f^{\frac{1}{2}}\|  \|  U^{1}\|  + 2\sum\limits_{n=2}^N k_n \|f^{n-1/2}\|\|U^{n-1/2}\|.
        \end{array}
     	\end{equation*}

  \vskip 0.2mm
     	By choosing a suitable $M$ such that $\|u^M\|=\max\limits_{0\leq n\leq N}{\|U^n\|}$, we can yield
     	\begin{equation}\label{eq4.8}
        \begin{array}{ll}
     	   \|U^M\|\leq  \|U^0\|  + 2\sum\limits_{n=1}^M k_n \|f^{n-1/2}\|
        \leq  \|U^0\| + 2\sum\limits_{n=1}^N k_n \|f^{n-1/2}\|.
        \end{array}
     	\end{equation}
     	We finish the proof.
     \end{proof}

     \subsection{Convergence analysis}

     \vskip 0.2mm
     Here, we consider the convergence of the second-order implicit difference scheme \eqref{eq3.19}-\eqref{eq3.22}.

    \vskip 0.2mm
     Denote
     $$\left.
     \begin{array}{ll}
     e_j^n=u_j^n-U_j^n,   \qquad 0\leq j\leq J,   \qquad 0\leq n\leq N.
     \end{array}
     \right.$$

    \vskip 0.2mm
     By subtracting \eqref{eq3.19}-\eqref{eq3.22} from \eqref{eq3.14}-\eqref{eq3.15}, \eqref{eq3.17}-\eqref{eq3.18}, respectively, we obtain the error equations as follows
 \begin{equation}\label{eq4.9}
 \begin{split}
   \delta_te_j^{\frac{1}{2}} & -  \nabla e^{1}_j \Delta_x  e^{1}_j = I^{(\alpha)} \delta_x^2e_j^{1/2}
    + \sum\limits_{m=1}^{4} (R_m)_j^{1/2},  \qquad 1\leq j \leq J-1,
 \end{split}
  \end{equation}
 \begin{equation}\label{eq4.10}
 \begin{split}
   \delta_te_j^{n-\frac{1}{2}} & - \nabla e^{n-\frac{1}{2}}_j \Delta_x  e^{n-\frac{1}{2}}_j  = I^{(\alpha)}\delta_x^2e_j^{n-\frac{1}{2}}  + \sum\limits_{m=1}^{4} (R_m)_j^{n-1/2}, \\
   & \qquad 1\leq j \leq J-1, \qquad 2\leq n\leq N,
 \end{split}
  \end{equation}
     \begin{equation}\label{eq4.11}
     \begin{array}{ll}
     e_j^0=0, \qquad 0\leq j\leq J,  \qquad   e_0^n=e_J^n=0, \qquad 1\leq n\leq N,
     \end{array}
     \end{equation}
from which,
  $$\left.
  \begin{array}{ll}
  \qquad (R_3)^{1/2}_{j} = -\frac{1}{6h}\Big[  u_j^1\Delta e_j^1+\Delta (e_j^1u_j^1) \Big],     \\
  \qquad (R_3)^{n-1/2}_{j} = -\frac{1}{6h}\Big[ u_j^{n-1/2}\Delta e_j^{n-1/2}+\Delta (e_j^{n-1/2}u_j^{n-1/2})\Big],   \qquad 2\leq n \leq N,  \\
  \qquad (R_4)^{1/2}_{j} = -\frac{1}{6h}\Big[ e_j^1\Delta u_j^1+\Delta (e_j^1u_j^1)\Big],  \\
  \qquad (R_4)^{n-1/2}_{j} = -\frac{1}{6h}\Big[ e_j^{n-1/2}\Delta u_j^{n-1/2}+\Delta (e_j^{n-1/2}u_j^{n-1/2})\Big],  \qquad 2\leq n \leq N.
  \end{array}
  \right. $$

  \vskip 0.2mm
 In order to obtain the convergence, we present the following lemmas.

   \vskip 0.2mm
   \begin{lemma}\label{lemma4.5} If satisfying the conditions
   \begin{equation}\label{eq4.12}
  \begin{split}
      t\|u_{xx}'(\cdot,t)\| +  t^2\|u_{xx}''(\cdot,t)\| \leq \mathcal{M} t^{\sigma-1},
  \end{split}
  \end{equation}
    \begin{equation}\label{eq4.13}
  \begin{split}
       t\| f'(\cdot,t)\| +  t^2\| f''(\cdot,t)\| \leq \mathcal{M} t^{\sigma-1}, \qquad  \sigma >0,
  \end{split}
  \end{equation}
 then we have
   $$\left.
   \begin{array}{ll}
      \sum\limits_{n=1}^N k_n \left( \left\|(R_1)^{n-\frac{1}{2}}\right\| + \left\|(R_2)^{n-\frac{1}{2}}\right\|  \right)  \leq  C\Big(\mathcal{K}_{L}+h^2\Big),
    \end{array}
   \right.
   $$
where
      \begin{equation*}
     \begin{split}
       \mathcal{K}_{L} :=
        \begin{cases}
          k^{\gamma\sigma}, & \mbox{if } 1\leq \gamma < \frac{2}{\sigma}, \\
          k^{2}\log(t_N/t_1), & \mbox{if }  \gamma = \frac{2}{\sigma}, \\
          k^2, & \mbox{if } \gamma > \frac{2}{\sigma}.
        \end{cases}
     \end{split}
     \end{equation*}
    \end{lemma}

  \begin{proof} With assumptions \eqref{eq4.12} and \eqref{eq4.13}, and from \cite[Corollary 4.3]{McLean}, we can get
     \begin{equation}\label{eq4.14}
     \begin{split}
            \sum\limits_{n=1}^N k_n\left\|(R_1)^{n-1/2}\right\| & \leq  \sum\limits_{n=1}^N k_n\left\|(R_{11})^{n-1/2}\right\| + \sum\limits_{n=1}^N k_n\left\|(R_{12})^{n-1/2}\right\| \\
            & \leq C_{\alpha,\gamma,\sigma,T}\mathcal{M}\times
       \begin{cases}
          k^{\gamma\sigma}, & \mbox{if } 1\leq \gamma < 2/\sigma, \\
          k^{2}\log(t_N/t_1), & \mbox{if }\gamma = 2/\sigma, \\
          k^2, & \mbox{if } \gamma > 2/\sigma.
        \end{cases}
     \end{split}
     \end{equation}
In addition, using \eqref{eq3.16}, we have
 \begin{equation}\label{eq4.15}
 \begin{split}
         \sum\limits_{n=1}^N & k_n\left\|(R_2)^{n-1/2}\right\|  =  k_1 \left\|(R_2)^{1/2}\right\|  + \sum\limits_{n=2}^N k_n\left\|(R_2)^{n-1/2}\right\|  \\
          & \leq Ck_1\left(k_1+h^2+ \left( \frac{t_1^{\alpha}-t_{0}^{\alpha}}{k_1\Gamma(\alpha+1)}\right)   h^2\right)+  C \sum\limits_{n=2}^N k_n \left(k_n^2 + h^2 + \left( \frac{t_n^{\alpha}-t_{n-1}^{\alpha}}{k_n\Gamma(\alpha+1)}\right)   h^2 \right) \\
          & \leq C(T)(k^2+h^2) + C \left(  \frac{t_1^{\alpha}-t_{0}^{\alpha}}{\Gamma(\alpha+1)} + \frac{t_N^{\alpha}-t_{1}^{\alpha}}{\Gamma(\alpha+1)} \right)h^2.
 \end{split}
  \end{equation}
Combining \eqref{eq4.14} and \eqref{eq4.15}, we finish the proof.
 \end{proof}

 \vskip 0.2mm
  \begin{lemma}\label{lemma4.6} If $u_0=u_J=0$ and $e_0=e_J=0$, it holds that
 $$\left.
  \begin{array}{ll}
        \langle (R_3)^{1/2} , e^1\rangle =0,  \quad  \langle (R_3)^{n-1/2} , e^{n-1/2}\rangle =0, \quad 2\leq n \leq N.
  \end{array}
 \right.
 $$
 \end{lemma}
 \begin{proof} The proof can be finished by Lemma \ref{lemma4.3}.
 \end{proof}

 \vskip 0.2mm
  \begin{lemma}\label{lemma4.7} If $u_0=u_J=0$ and $e_0=e_J=0$, then we get
 $$\left.
  \begin{array}{ll}
 \left|\langle (R_4)^{1/2} , e^1\rangle \right| \leq \frac{c_0}{2}  \| e^1\|^2,  \quad  \left|\langle (R_4)^{n-1/2} , e^{n-1/2}\rangle \right| \leq \frac{c_0}{2}  \| e^{n-1/2}\|^2, \quad 2\leq n \leq N.
  \end{array}
 \right.
 $$
 \end{lemma}

 \begin{proof} From Lemma \ref{lemma4.2}, we have
\begin{equation*}
  \begin{array}{ll}
  \langle\Delta (e^mu^m),e^m\rangle=\frac{1}{2}\langle T_+e^m\Delta_+u^m+T_-e^m\Delta_-u^m,e^m\rangle.
  \end{array}
\end{equation*}

  \vskip 0.2mm
 With conditions $u_0=u_J=0$, $e_0=e_J=0$ and using Cauchy-Schwarz inequality, we yield that
\begin{equation*}
  \begin{array}{ll}
  |\langle T_+e^m\Delta_+u^m,e^m\rangle|\leq \| \Delta_+u^m \|_{\infty}|\langle T_+e^m,e^m\rangle|\\
  =\| \Delta_+u^m \|_{\infty}\left|\sum\limits_{s=1}^{J-1}h e_{s+1}^me_s^m\right|
  \leq\| \Delta_+u^m \|_{\infty}\sum\limits_{s=1}^{J-1}\frac{h}{2}((e_{s+1}^m)^2+(e_s^m)^2)\\
  =\frac{1}{2}\| \Delta_+u^m \|_{\infty}\left(\sum\limits_{s=1}^{J-1}h{(e_{s+1}^m)}^2 +\| e^m\|^2\right)
  \leq\frac{1}{2}\| \Delta_+u^m \|_{\infty}\left(\sum\limits_{s=2}^Jh{(e_s^m)}^2+\| e^m\|^2 \right)\\
  \leq\| \Delta_+u^m \|_{\infty}\| e^m\|^2.
  \end{array}
\end{equation*}
  \vskip 0.2mm
 Similarly, it holds that
\begin{equation*}
  \begin{array}{ll}
  |\langle T_-e^m\Delta_-u^m,e^m\rangle|\leq \| \Delta_-u^m \|_{\infty}\| e^m\|^2.
  \end{array}
\end{equation*}

  \vskip 0.2mm
 Thus, we can get
\begin{equation*}
  \begin{split}
 & \Big|-\langle e^m\Delta u^m+\Delta (e^mU^m),e^m\rangle \Big| \\
 & = \left|\langle e^m\Delta u^m,e^m\rangle + \frac{1}{2}\langle T_+e^m\Delta_+u^m+T_-e^m\Delta_-u^m,e^m\rangle  \right|\\
 & \leq \|\Delta u^m \|_{\infty}\| e^m\|^2   + \frac{1}{2}\left(\| \Delta_+u^m \|_{\infty}+\| \Delta_-u^m \|_{\infty}\right)\| e^m\|^2\\
 & \leq  3c_0 h \| e^m\|^2,
  \end{split}
\end{equation*}
which finishes the proof.
 \end{proof}

   \vskip 0.2mm
  \begin{theorem}\label{th4.3} Let $\{u^n\}_{n=0}^{N}\in C_{x}^{4}([0,L])$ and $\{U^n\}_{n=0}^{N}$ be the solutions of \eqref{eq1.1}-\eqref{eq1.3} and \eqref{eq3.19}-\eqref{eq3.21}, respectively. Assume that $f^{n-1/2}$ is selected to meet that \eqref{eq2.7} and \eqref{eq2.8} hold. If the exact solution and the forcing term satisfy
  \begin{equation}\label{eq4.16}
  \begin{split}
      t\|u_{xx}'(\cdot,t)\| +  t^2\|u_{xx}''(\cdot,t)\| \leq \mathcal{M} t^{\sigma-1},
  \end{split}
  \end{equation}
  \begin{equation}\label{eq4.17}
  \begin{split}
       t\| f'(\cdot,t)\| +  t^2\| f''(\cdot,t)\| \leq \mathcal{M} t^{\sigma-1}, \quad \sigma>0,
  \end{split}
  \end{equation}
 respectively for $t>0$, then for $t_n\in [0,T]$ and $\gamma\geq 1$, it holds that
  \begin{equation*}
  \begin{split}
    \max\limits_{1\leq n \leq N}  \| U^n - u^n \| & \leq  C(T) \Big( h^2 + \mathcal{K}_{L} \Big),
  \end{split}
  \end{equation*}
from which,
      \begin{equation*}
     \begin{split}
       \mathcal{K}_{L} :=
        \begin{cases}
          k^{\gamma\sigma}, & \mbox{if } 1\leq \gamma < \frac{2}{\sigma}, \\
          k^{2}\log(t_N/t_1), & \mbox{if }\gamma = \frac{2}{\sigma}, \\
          k^2, & \mbox{if } \gamma > \frac{2}{\sigma}.
        \end{cases}
     \end{split}
     \end{equation*}
 \end{theorem}

   \vskip 0.2mm
   \begin{proof} First, taking the inner product of \eqref{eq3.19}-\eqref{eq3.20} with $k_1e^{1}$ and $k_ne^{n-1/2}$, respectively, summing up for $n$ from 1 to $N$, and utilizing Lemmas \ref{lemma4.3} and \ref{lemma4.1}, we have
     \begin{equation}\label{eq4.18}
      \begin{split}
       k_1\langle \delta_te^{\frac{1}{2}} ,  e^{1}\rangle & +  \sum\limits_{n=2}^Nk_n \langle \delta_te^{n-\frac{1}{2}}, e^{n-\frac{1}{2}} \rangle
        \\
        &= k_1 \left\langle  \sum\limits_{m=1}^4 (R_m)^{\frac{1}{2}} ,  e^{1} \right\rangle + \sum\limits_{n=2}^N k_n \left\langle \sum\limits_{m=1}^4 (R_m)^{n-1/2}, e^{n-1/2}\right\rangle \\
        & - \left(  k_1\langle  I^{(\alpha)}\delta_xe^{\frac{1}{2}}, \delta_xe^{1} \rangle +  \sum\limits_{n=2}^N k_n \langle  I^{(\alpha)}\delta_xe^{n-\frac{1}{2}}, \delta_xe^{n-1/2} \rangle \right),
     \end{split}
    \end{equation}
from which, we use Lemma \ref{lemma4.6} and Lemma \ref{lemma4.7} to get
     \begin{equation}\label{eq4.19}
      \begin{split}
       k_1\langle \delta_te^{\frac{1}{2}} ,  e^{1}\rangle & +  \sum\limits_{n=2}^Nk_n \langle \delta_te^{n-\frac{1}{2}}, e^{n-\frac{1}{2}} \rangle
        \\
        &\leq  k_1 \left\langle  \sum\limits_{m=1}^2 (R_m)^{\frac{1}{2}} ,  e^{1} \right\rangle + \sum\limits_{n=2}^N k_n \left\langle \sum\limits_{m=1}^2 (R_m)^{n-1/2}, e^{n-1/2}\right\rangle \\
        & + \frac{c_0}{2}k_1 \|e^1\|^2 + \frac{c_0}{2}\sum\limits_{n=2}^N k_n \|e^{n-1/2}\|^2 \\
        & - \left(  k_1\langle  I^{(\alpha)}\delta_xe^{\frac{1}{2}}, \delta_xe^{1} \rangle +  \sum\limits_{n=2}^N k_n \langle  I^{(\alpha)}\delta_xe^{n-\frac{1}{2}}, \delta_xe^{n-1/2} \rangle \right).
     \end{split}
    \end{equation}

   \vskip 0.2mm
Then using \eqref{eq4.6}, \eqref{eq4.7} and the Cauchy-Schwarz inequality, we have
     \begin{equation}\label{eq4.20}
     \begin{split}
     	         \frac{1}{2} \Big(\|e^{N}\|^2-\|e^{0}\|^2 \Big) & \leq k_1 \|(R_1)^{\frac{1}{2}}\| \|e^1\| + \sum\limits_{n=2}^N k_n \|(R_1)^{n-\frac{1}{2}}\| \|e^{n-\frac{1}{2}}\| \\
     & + k_1  \|(R_2)^{\frac{1}{2}}\| \|e^1\| + \sum\limits_{n=2}^N k_n\|(R_2)^{n-\frac{1}{2}}\|  \|e^{n-\frac{1}{2}}\| \\
     & + \frac{c_0}{2}k_1 \|e^1\|^2 + \frac{c_0}{2}\sum\limits_{n=2}^N k_n \left\|e^{n-\frac{1}{2}}\right\|^2.
     \end{split}
     \end{equation}

   \vskip 0.2mm
   By taking an appropriate $K$ such that $\|e^K\|=\max\limits_{0\leq n\leq N}{\|e^n\|}$ and using \eqref{eq4.11}, we have
     \begin{equation}\label{eq4.21}
     \begin{split}
     	       \|e^{K}\|^2 & \leq  2\sum\limits_{n=1}^K k_n \|(R_1)^{n-\frac{1}{2}}\| \|e^{K}\| + 2\sum\limits_{n=1}^K k_n \|(R_2)^{n-\frac{1}{2}}\| \|e^{K}\| \\
     & + c_0k_1 \|e^1\|\|e^K\| + c_0 \sum\limits_{n=2}^K k_n \left\|e^{n-\frac{1}{2}}\right\|\|e^K\|,
     \end{split}
     \end{equation}
   therefore,
     \begin{equation}\label{eq4.22}
     \begin{split}
     	       \|e^{K}\| & \leq  2\sum\limits_{n=1}^K k_n \|(R_1)^{n-\frac{1}{2}}\| + 2 \sum\limits_{n=1}^K k_n\|(R_2)^{n-\frac{1}{2}}\|   \\
     & + c_0k_1 \|e^1\| + \frac{c_0}{2} \sum\limits_{n=2}^K k_n \left(\|e^{n-1}\|+ \|e^{n}\|\right).
     \end{split}
     \end{equation}

This naturally can obtain
   \begin{equation}\label{eq4.23}
     \begin{split}
     	   \left( 1- \frac{c_0}{2}k_N \right)    \|e^{N}\| & \leq 2 \sum\limits_{n=1}^N k_n \|(R_1)^{n-\frac{1}{2}}\|  + 2 \sum\limits_{n=1}^N k_n \|(R_2)^{n-\frac{1}{2}}\| \\
     &  + \frac{3c_0}{2} \sum\limits_{n=1}^{N-1} k_n \|e^{n}\|.
     \end{split}
     \end{equation}

   \vskip 0.2mm
   When $k_N\leq \frac{1}{c_0}$, we utilize the discrete Gr\"{o}nwall inequality to yield
   \begin{equation}\label{eq4.24}
     \begin{split}
     	    \left\|e^{N}\right\|    & \leq  C\exp\left(\sum\limits_{n=1}^{N} k_n \right) \left(2 \sum\limits_{n=1}^N k_n \|(R_1)^{n-\frac{1}{2}}\|  + 2 \sum\limits_{n=1}^N k_n \|(R_2)^{n-\frac{1}{2}}\| \right),
     \end{split}
     \end{equation}
from which, we can complete the proof by Lemma \ref{lemma4.5}.
 \end{proof}

 \subsection{Uniqueness of numerical solutions}
   \vskip 0.2mm
    Next, based on the hypothesis of Theorem \ref{th4.3}, we will prove the uniqueness of numerical solutions for second-order implicit difference scheme \eqref{eq3.19}-\eqref{eq3.22}.

   \vskip 0.2mm
   \begin{theorem}\label{th4.4} If $h$ is sufficiently small, $k=o(h^{\frac{1}{2}})$ and $\gamma = \frac{2}{\sigma}$, then second-order implicit difference scheme \eqref{eq3.19}-\eqref{eq3.22} has a unique solution.
   \end{theorem}

  \vskip 0.2mm
  \begin{proof} Let $U^{n}\in \mathbb{R}^{J-1}$ and $V^{n}\in \mathbb{R}^{J-1}$, $ 0\leq n\leq N$, be the solutions of \eqref{eq3.19}-\eqref{eq3.22} which satisfy $U^{0}=V^{0}$. From \cite[p.~29]{Lopez-Marcos}, we can demonstrate similarly that \eqref{eq3.19} has a unique solution $U^1$. Then we suppose that $U^{m}=V^{m}$, $1\leq m\leq n-1$, $2\leq n \leq N$.

   \vskip 0.2mm
   Next, we only need to show $U^{n}=V^{n}$ for \eqref{eq3.20}, with $2\leq n \leq N$. Firstly, by analyzing \eqref{eq3.20}, we get
   \begin{equation}\label{eq4.25}
    \begin{split}
   \left(\delta_tU_j^{n-\frac{1}{2}}- \delta_tV_j^{n-\frac{1}{2}} \right) & + \frac{1}{6h}\left(U_j^{n-\frac{1}{2}}\Delta U_j^{n-\frac{1}{2}}+\Delta (U_j^{n-\frac{1}{2}})^2 - V_j^{n-\frac{1}{2}}\Delta V_j^{n-\frac{1}{2}} - \Delta (V_j^{n-\frac{1}{2}})^2 \right)  \\ & = I^{(\alpha)}\delta_x^2\left(U_j^{n-\frac{1}{2}}-V_j^{n-\frac{1}{2}}\right),  \qquad 1\leq j\leq J-1.
    \end{split}
   \end{equation}

   \vskip 0.2mm
   Then, taking the inner product of \eqref{eq4.25} with $U^{n-\frac{1}{2}}-V^{n-\frac{1}{2}}$ and employing Lemma \ref{lemma4.1}, we yield
   \begin{equation*}
    \begin{split}
        &\frac{1}{2k_n}( \|U^n - V^n\|^2 - \|U^{n-1} - V^{n-1}\|^2 )  \\
        &\leq - \frac{1}{6h} \left\langle U^{n-\frac{1}{2}}\Delta U^{n-\frac{1}{2}}+\Delta (U^{n-\frac{1}{2}})^2 - V^{n-\frac{1}{2}}\Delta V^{n-\frac{1}{2}} - \Delta (V^{n-\frac{1}{2}})^2, U^{n-\frac{1}{2}}-V^{n-\frac{1}{2}} \right\rangle   \\
         & =  - \frac{1}{6h} \Big\langle U^{n-\frac{1}{2}}\Delta(U^{n-\frac{1}{2}}-V^{n-\frac{1}{2}}) + (U^{n-\frac{1}{2}}-V^{n-\frac{1}{2}})\Delta V^{n-\frac{1}{2}} \\
         & \qquad\qquad\qquad + \Delta(U^{n-\frac{1}{2}}-V^{n-\frac{1}{2}})(U^{n-\frac{1}{2}}+V^{n-\frac{1}{2}})   , U^{n-\frac{1}{2}} - V^{n-\frac{1}{2}} \Big\rangle,
    \end{split}
   \end{equation*}
   and using Lemma \ref{lemma4.4}, we obtain
   \begin{equation}\label{eq4.26}
    \begin{split}
     & \|U^n - V^n\|^2  \leq \|U^{n-1} - V^{n-1}\|^2   \\
       &  - \frac{k_n}{3h} \left\langle  (U^{n-\frac{1}{2}}-V^{n-\frac{1}{2}})\Delta V^{n-\frac{1}{2}} + \Delta(V^{n-\frac{1}{2}}(U^{n-\frac{1}{2}}-V^{n-\frac{1}{2}}))   , U^{n-\frac{1}{2}} - V^{n-\frac{1}{2}} \right\rangle  \\
    & \leq \frac{k_n}{3h} \left|\left\langle  (U^{n-\frac{1}{2}}-V^{n-\frac{1}{2}})\Delta V^{n-\frac{1}{2}} + \Delta(V^{n-\frac{1}{2}}(U^{n-\frac{1}{2}}-V^{n-\frac{1}{2}}))   , U^{n-\frac{1}{2}} - V^{n-\frac{1}{2}} \right\rangle\right|.
    \end{split}
   \end{equation}

   \vskip 0.2mm
  Using the Cauchy-Schwarz inequality, then we get
   \begin{equation}\label{eq4.27}
      \begin{split}
     \Theta^n& \equiv \left|\left\langle  (U^{n-\frac{1}{2}}-V^{n-\frac{1}{2}})\Delta V^{n-\frac{1}{2}} + \Delta(V^{n-\frac{1}{2}}(U^{n-\frac{1}{2}}-V^{n-\frac{1}{2}}))   , U^{n-\frac{1}{2}} - V^{n-\frac{1}{2}} \right\rangle\right|   \\
    &\leq  \left\|\Delta V^{n-\frac{1}{2}}\right\|_{\infty} \|U^{n-\frac{1}{2}} - V^{n-\frac{1}{2}}\|^2  +  \frac{1}{2}\left( \|\Delta_{+} V^{n-\frac{1}{2}}\|_{\infty}  + \|\Delta_{-} V^{n-\frac{1}{2}}\|_{\infty} \right)\|U^{n-\frac{1}{2}} - V^{n-\frac{1}{2}}\|^2,
    \end{split}
   \end{equation}
   from which, applying the triangle inequality, we obtain
   \begin{equation*}
     \begin{split}
           \left\|\Delta V^{n-\frac{1}{2}}\right\|_{\infty} &= \max \limits_{1\leq j \leq J-1}\left\{  \left|V^{n-\frac{1}{2}}_{j+1} - V^{n-\frac{1}{2}}_{j-1}\right| \right\}  \\
           &\leq  \max \limits_{1\leq j \leq J-1}\left\{  \left|V^{n-\frac{1}{2}}_{j+1} - v^{n-\frac{1}{2}}_{j+1}\right| + \left|v^{n-\frac{1}{2}}_{j+1} - v^{n-\frac{1}{2}}_{j-1}\right| + \left|v^{n-\frac{1}{2}}_{j-1} - V^{n-\frac{1}{2}}_{j-1}\right| \right\}  \\
          & \leq  2 \left\|V^{n-\frac{1}{2}} - v^{n-\frac{1}{2}}\right\|_{\infty} + Ch  \leq  2 h^{-\frac{1}{2}}\left\|V^{n-\frac{1}{2}} - v^{n-\frac{1}{2}}\right\| +  Ch  \\
          & \leq h^{-\frac{1}{2}}\left\|V^{n} - v^{n}\right\| +  Ch,
    \end{split}
   \end{equation*}
  hence, we employ Theorem \ref{th4.3} with $\gamma = \frac{2}{\sigma}$, then
   \begin{equation}\label{eq4.28}
     \begin{split}
          \Theta^n \leq  C \Big[ h^{-\frac{1}{2}}(k^2+h^2) +  Ch \Big] \|U^n - V^n\|^2.
    \end{split}
   \end{equation}

    \vskip 0.2mm
    By substituting \eqref{eq4.28} into \eqref{eq4.26}, we can get
    \begin{equation}\label{eq4.29}
    \begin{array}{ll}
             \|U^n - V^n\|^2  \leq  C k/h \Big[ h^{-\frac{1}{2}}(k^2+h^2) +  Ch \Big] \|U^n - V^n\|^2, \qquad 2\leq n \leq N.
    \end{array}
    \end{equation}

    \vskip 0.2mm
    With \eqref{eq4.29}, we obtain $\|U^n - V^n\|^2=0$ for $k=o(h^{\frac{1}{2}})$ as $h \rightarrow 0$. Within this assumption condition, the proof is finished.
    \end{proof}

\section{Numerical experiment}

  \vskip 0.2mm
  In this section, we set $L=T=1$. We use the second-order implicit difference scheme \eqref{eq3.19}-\eqref{eq3.22} to solve problem \eqref{eq1.1}-\eqref{eq1.3}, based on an iterative algorithm \cite[Algorithm 1]{Qiu1}, with $MaxStep=300$ and $eps=1$e-6. All experiments are carried out by MATLAB (R2014b) on a Windows 10 with CPU (2.20 GHz), RAM (8.0 GB).

  \vskip 0.2mm
  Then define the following $L^2$ errors
   $$\left.
 \begin{array}{ll}
    E_{CN}(N,J) := \left\|U^N - u^N \right\|
 \end{array}
 \right.
 $$
 and the convergence orders
 $$\left.
 \begin{array}{ll}
  Rate^t :=\log_{2}\left(\frac{E_{CN}(N,J)}{E_{CN}(2N,J)}\right), \quad
  Rate^x :=\log_{2}\left(\frac{E_{CN}(N,J)}{E_{CN}(N,2J)}\right),
 \end{array}
 \right.
 $$
 respectively. In addition, we define
   \begin{equation}\label{eq5.1}
     \begin{split}
          f^{n-\frac{1}{2}} := \frac{1}{k_n} \int_{t_{n-1}}^{t_n} f(x,t)dt
     \end{split}
     \end{equation}
to replace $f^{n-1/2} =f(t_{n-1/2})$.

\vskip 1mm
  \textbf{Example 1.} First, we give the exact solution
   \begin{equation}\label{eq5.2}
     \begin{split}
          u(x,t) = \sin \pi x  - \frac{t^{\alpha+1}}{\Gamma(\alpha+2)} \sin 2\pi x,
     \end{split}
     \end{equation}
 hence, $u_0(x)=\sin \pi x $ and
   \begin{equation}\label{eq5.3}
     \begin{split}
          f(x,t) &= \frac{\pi^2t^{\alpha}}{\Gamma(\alpha+1)}  \sin \pi x   - \left( \frac{4\pi^2t^{2\alpha+1}}{\Gamma(2\alpha+2)}  + \frac{t^{\alpha}}{\Gamma(\alpha+1)}   \right)\sin 2\pi x \\
          & + \left(  \sin \pi x  - \frac{t^{\alpha+1}}{\Gamma(\alpha+2)} \sin 2\pi x \right)  \left( \pi \cos \pi x  - \frac{2\pi t^{\alpha+1}}{\Gamma(\alpha+2)} \cos 2\pi x\right).
     \end{split}
     \end{equation}

\vskip 0.2mm
Therefore, for $t>0$, we have \cite{McLean}
   \begin{equation}\label{eq5.4}
     \begin{split}
          t \|u_{xx}'(\cdot, t)\| + t^2\|u_{xx}''(\cdot, t)\|\leq  \mathcal{M}t^{\alpha+1},  \quad t \|f'(\cdot, t)\| + t^2\|f''(\cdot, t)\|\leq  \mathcal{M}t^{\alpha},
     \end{split}
     \end{equation}
which implies that the regularity conditions \eqref{eq4.16}-\eqref{eq4.17} of Theorem \ref{th4.3} are satisfied with $\sigma=\alpha+1$. In Table \ref{tb1}, we show some results when taking $f^{n-\frac{1}{2}}=\frac{1}{2}(f^{n}+f^{n-1})$. As predicted, the temporal convergence orders approximate
$k^{1+\alpha}$ under a uniform mesh, i.e., $\gamma=1$. However, from Table \ref{tb2}, the convergence orders can improve to $k^{2}$ when $\gamma=\frac{2}{\alpha+1}$ (cf.~Theorem \ref{th4.3}).

\vskip 0.2mm
Then, from Table \ref{tb3}, we illustrate the effectiveness of eliminating the errors caused by the approximation \eqref{eq2.6}, employing \eqref{eq5.1} to replace $f^{n-\frac{1}{2}}=f(\cdot,t_{n-\frac{1}{2}})$, which means that the singular behaviour of $f$ on longer plays an important role, and Theorem \ref{th4.3} holds with $\sigma=\alpha+2$ (namely $\gamma\geq \frac{2}{\sigma}$), thus we predict the second-order convergence for time with $\gamma=1$. In addition, the second-order accuracy for time can be observed in Table \ref{tb4}, by fixing $J=1024$ and $\gamma=\frac{2}{\alpha+1}$ when $f^{n-1/2}$ is given by \eqref{eq5.1}.

\vskip 0.2mm
When $f^{n-\frac{1}{2}}$ is given by \eqref{eq5.1} and fixing $\alpha=0.05,0.35,0.65,0.95$, respectively, it can be observed clearly from Tables \ref{tb5} and \ref{tb6} that the implicit difference scheme is convergent to the order 2 for space as expected, which is in accordance with our theoretical analysis.

 \begin{table}
   \centering\small
   \caption{The $L^2$ errors, temporal convergence orders and CPU time (seconds) for $J=1024$
   and a uniform mesh ($\gamma=1$) with $f^{n-1/2}=\frac{1}{2}(f^{n}+f^{n-1})$.}\label{tb1}
 \begin{tabular}{c c c c c c c}\hline
  $\alpha$  &\qquad $ N $  &\qquad  $E_{CN}$    &\qquad  $Rate^t$   &\qquad  CPU(s)   \\ \hline
                  &\qquad $ 8$    &\qquad  2.0116e-3  &\qquad    *        &\qquad  4.61      \\
                  &\qquad $ 16$   &\qquad  9.6258e-4  &\qquad    1.06     &\qquad  8.20     \\
   $0.05$         &\qquad $ 32$   &\qquad  4.6360e-4  &\qquad    1.05     &\qquad  15.41     \\
                  &\qquad $ 64$   &\qquad  2.2437e-4  &\qquad    1.05     &\qquad  31.90    \\
                  &\qquad $ 8$    &\qquad  5.2489e-3  &\qquad    *        &\qquad  4.66      \\
                  &\qquad $ 16$   &\qquad  2.1570e-3  &\qquad    1.28     &\qquad  8.85     \\
   $0.25$         &\qquad $ 32$   &\qquad  9.0520e-4  &\qquad    1.25     &\qquad  16.88     \\
                  &\qquad $ 64$   &\qquad  3.8131e-4  &\qquad    1.25     &\qquad  33.63    \\
                  &\qquad $ 8$    &\qquad  7.8050e-3  &\qquad    *        &\qquad  5.27      \\
                  &\qquad $ 16$   &\qquad  2.5050e-3  &\qquad    1.64     &\qquad  10.09     \\
   $0.60$         &\qquad $ 32$   &\qquad  8.3659e-4  &\qquad    1.58     &\qquad  18.13     \\
                  &\qquad $ 64$   &\qquad  2.8073e-4  &\qquad    1.58     &\qquad  33.90    \\
                  &\qquad $ 8$    &\qquad  1.8670e-3  &\qquad    *        &\qquad  6.34      \\
                  &\qquad $ 16$   &\qquad  4.1218e-4  &\qquad    2.18     &\qquad  10.68     \\
   $0.95$         &\qquad $ 32$   &\qquad  1.0774e-4  &\qquad    1.94     &\qquad  17.76     \\
                  &\qquad $ 64$   &\qquad  3.0643e-5  &\qquad    1.81     &\qquad  31.09    \\
     \hline
   \end{tabular}
 \end{table}

  \begin{table}
   \centering\small
   \caption{The $L^2$ errors, temporal convergence orders and CPU time (seconds) for $J=1024$
   and a graded mesh ($\gamma=\frac{2}{\alpha+1}$) when $f^{n-1/2}=\frac{1}{2}(f^{n}+f^{n-1})$.}\label{tb2}
 \begin{tabular}{c c c c c c c}\hline
  $\alpha$  &\qquad $ N $  &\qquad  $E_{CN}$    &\qquad  $Rate^t$   &\qquad  CPU(s)   \\ \hline
                  &\qquad $ 8$    &\qquad  4.8808e-4  &\qquad    *        &\qquad  4.21      \\
                  &\qquad $ 16$   &\qquad  1.4323e-4  &\qquad    1.77     &\qquad  7.71     \\
   $0.05$         &\qquad $ 32$   &\qquad  4.0144e-5  &\qquad    1.84     &\qquad  14.74     \\
                  &\qquad $ 64$   &\qquad  1.0342e-5  &\qquad    1.96     &\qquad  32.05    \\
                  &\qquad $ 8$    &\qquad  1.6004e-3  &\qquad    *        &\qquad  4.55      \\
                  &\qquad $ 16$   &\qquad  4.4928e-4  &\qquad    1.83     &\qquad  8.30     \\
   $0.25$         &\qquad $ 32$   &\qquad  1.2346e-4  &\qquad    1.86     &\qquad  16.16     \\
                  &\qquad $ 64$   &\qquad  3.3192e-5  &\qquad    1.90     &\qquad  33.07    \\
                  &\qquad $ 8$    &\qquad  4.3516e-3  &\qquad    *        &\qquad  5.53      \\
                  &\qquad $ 16$   &\qquad  1.1749e-3  &\qquad    1.89     &\qquad  9.79     \\
   $0.60$         &\qquad $ 32$   &\qquad  3.2463e-4  &\qquad    1.86     &\qquad  17.70     \\
                  &\qquad $ 64$   &\qquad  8.9770e-5  &\qquad    1.85     &\qquad  32.28    \\
                  &\qquad $ 8$    &\qquad  1.6900e-3  &\qquad    *        &\qquad  8.32      \\
                  &\qquad $ 16$   &\qquad  3.6973e-4  &\qquad    2.19     &\qquad  10.29     \\
   $0.95$         &\qquad $ 32$   &\qquad  9.6631e-5  &\qquad    1.94     &\qquad  16.67     \\
                  &\qquad $ 64$   &\qquad  2.7185e-5  &\qquad    1.83     &\qquad  32.56    \\
     \hline
   \end{tabular}
 \end{table}

 \begin{table}
   \centering\small
   \caption{The $L^2$ errors, temporal convergence orders and CPU time (seconds) for $J=1024$ and a uniform mesh ($\gamma=1$)
   when $f^{n-1/2}$ is presented by \eqref{eq5.1}.}\label{tb3}
 \begin{tabular}{c c c c c c c}\hline
  $\alpha$  &\qquad $ N $  &\qquad  $E_{CN}$    &\qquad  $Rate^t$   &\qquad  CPU(s)   \\ \hline
                  &\qquad $ 8$    &\qquad  4.9645e-4  &\qquad    *        &\qquad  4.38      \\
                  &\qquad $ 16$   &\qquad  1.2346e-4  &\qquad    2.01     &\qquad  10.50     \\
   $0.25$         &\qquad $ 32$   &\qquad  2.9472e-5  &\qquad    2.07     &\qquad  17.53     \\
                  &\qquad $ 64$   &\qquad  6.0076e-6  &\qquad    2.29     &\qquad  37.04    \\
                  &\qquad $ 8$    &\qquad  1.1946e-3  &\qquad    *        &\qquad  5.03      \\
                  &\qquad $ 16$   &\qquad  2.5727e-4  &\qquad    2.21     &\qquad  9.16     \\
   $0.50$         &\qquad $ 32$   &\qquad  6.3102e-5  &\qquad    2.03     &\qquad  16.63     \\
                  &\qquad $ 64$   &\qquad  1.4618e-5  &\qquad    2.11     &\qquad  34.55    \\
                  &\qquad $ 8$    &\qquad  1.9949e-3  &\qquad    *        &\qquad  5.83      \\
                  &\qquad $ 16$   &\qquad  4.9271e-4  &\qquad    2.02     &\qquad  10.19     \\
   $0.75$         &\qquad $ 32$   &\qquad  1.1761e-4  &\qquad    2.07     &\qquad  16.74     \\
                  &\qquad $ 64$   &\qquad  2.7211e-5  &\qquad    2.11     &\qquad  31.59    \\
                  &\qquad $ 8$    &\qquad  1.8929e-3  &\qquad    *        &\qquad  6.39      \\
                  &\qquad $ 16$   &\qquad  3.6419e-4  &\qquad    2.38     &\qquad  10.77     \\
   $0.95$         &\qquad $ 32$   &\qquad  8.1789e-5  &\qquad    2.15     &\qquad  17.34     \\
                  &\qquad $ 64$   &\qquad  1.9272e-5  &\qquad    2.08     &\qquad  31.79    \\
     \hline
   \end{tabular}
 \end{table}

  \begin{table}
   \centering\small
   \caption{The $L^2$ errors, temporal convergence orders and CPU time (seconds) for $J=1024$ and a graded mesh ($\gamma=\frac{2}{\alpha+1}$)
   when $f^{n-1/2}$ is presented by \eqref{eq5.1}.}\label{tb4}
 \begin{tabular}{c c c c c c c}\hline
  $\alpha$  &\qquad $ N $  &\qquad  $E_{CN}$    &\qquad  $Rate^t$   &\qquad  CPU(s)   \\ \hline
                  &\qquad $ 8$    &\qquad  2.4027e-3  &\qquad    *        &\qquad  4.36      \\
                  &\qquad $ 16$   &\qquad  6.3364e-4  &\qquad    1.92     &\qquad  8.38     \\
   $0.25$         &\qquad $ 32$   &\qquad  1.6421e-4  &\qquad    1.95     &\qquad  15.87     \\
                  &\qquad $ 64$   &\qquad  4.1219e-5  &\qquad    1.99     &\qquad  31.94    \\
                  &\qquad $ 8$    &\qquad  2.7254e-3  &\qquad    *        &\qquad  4.81      \\
                  &\qquad $ 16$   &\qquad  6.7966e-4  &\qquad    2.00     &\qquad  8.89     \\
   $0.50$         &\qquad $ 32$   &\qquad  1.7026e-4  &\qquad    2.00     &\qquad  16.37     \\
                  &\qquad $ 64$   &\qquad  4.1702e-5  &\qquad    2.03     &\qquad  32.92    \\
                  &\qquad $ 8$    &\qquad  2.3405e-3  &\qquad    *        &\qquad  5.48      \\
                  &\qquad $ 16$   &\qquad  5.9713e-4  &\qquad    1.97     &\qquad  9.69     \\
   $0.75$         &\qquad $ 32$   &\qquad  1.4693e-4  &\qquad    2.02     &\qquad  18.30     \\
                  &\qquad $ 64$   &\qquad  3.5482e-5  &\qquad    2.05     &\qquad  34.10    \\
                  &\qquad $ 8$    &\qquad  1.8367e-3  &\qquad    *        &\qquad  6.20      \\
                  &\qquad $ 16$   &\qquad  3.7228e-4  &\qquad    2.30     &\qquad  11.53     \\
   $0.95$         &\qquad $ 32$   &\qquad  8.8065e-5  &\qquad    2.08     &\qquad  16.67     \\
                  &\qquad $ 64$   &\qquad  2.1234e-5  &\qquad    2.05     &\qquad  30.99    \\
     \hline
   \end{tabular}
 \end{table}

 \begin{table}
   \centering\small
   \caption{The $L^2$ errors, spatial convergence orders and CPU time (seconds) for $N=256$ and a graded mesh ($\gamma=\frac{2}{\alpha+1}$)
   when $f^{n-1/2}=\frac{1}{2}(f^{n}+f^{n-1})$.}\label{tb5}
 \begin{tabular}{c c c c c c c}\hline
  $\alpha$  &\qquad $ J   $  &\qquad  $E_{CN}$    &\qquad  $Rate^x$   &\qquad  CPU(s)   \\ \hline
                  &\qquad $ 8$    &\qquad  3.7100e-2  &\qquad    *        &\qquad  0.31      \\
                  &\qquad $ 16$   &\qquad  9.0767e-3  &\qquad    2.03     &\qquad  0.32    \\
   $0.05$         &\qquad $ 32$   &\qquad  2.2566e-3  &\qquad    2.01     &\qquad  0.35     \\
                  &\qquad $ 64$   &\qquad  5.6294e-4  &\qquad    2.00     &\qquad  0.50    \\
                  &\qquad $ 8$    &\qquad  3.2425e-2  &\qquad    *        &\qquad  0.32      \\
                  &\qquad $ 16$   &\qquad  7.9333e-3  &\qquad    2.03     &\qquad  0.36    \\
   $0.35$         &\qquad $ 32$   &\qquad  1.9727e-3  &\qquad    2.01     &\qquad  0.38     \\
                  &\qquad $ 64$   &\qquad  4.9245e-4  &\qquad    2.00     &\qquad  0.54    \\
                  &\qquad $ 8$    &\qquad  2.7412e-2  &\qquad    *        &\qquad  0.32      \\
                  &\qquad $ 16$   &\qquad  6.7176e-3  &\qquad    2.03     &\qquad  0.35    \\
   $0.65$         &\qquad $ 32$   &\qquad  1.6726e-3  &\qquad    2.01     &\qquad  0.38     \\
                  &\qquad $ 64$   &\qquad  4.1903e-4  &\qquad    2.00     &\qquad  0.56    \\
                  &\qquad $ 8$    &\qquad  2.7276e-2  &\qquad    *        &\qquad  0.32      \\
                  &\qquad $ 16$   &\qquad  6.5935e-3  &\qquad    2.05     &\qquad  0.33    \\
   $0.95$         &\qquad $ 32$   &\qquad  1.6360e-3  &\qquad    2.01     &\qquad  0.38     \\
                  &\qquad $ 64$   &\qquad  4.0899e-4  &\qquad    2.00     &\qquad  0.57    \\
     \hline
   \end{tabular}
 \end{table}

 \begin{table}
   \centering\small
   \caption{The $L^2$ errors, spatial convergence orders and CPU time (seconds) for $N=256$ and a graded mesh ($\gamma=\frac{2}{\alpha+1}$)
    when $f^{n-1/2}$ is presented by \eqref{eq5.1}.}\label{tb6}
 \begin{tabular}{c c c c c c c}\hline
  $\alpha$  &\qquad $ J $  &\qquad  $E_{CN}$    &\qquad  $Rate^x$   &\qquad  CPU(s)   \\ \hline
                  &\qquad $ 8$    &\qquad  3.7100e-2  &\qquad    *        &\qquad  0.30      \\
                  &\qquad $ 16$   &\qquad  9.0762e-3  &\qquad    2.03     &\qquad  0.31    \\
   $0.05$         &\qquad $ 32$   &\qquad  2.2562e-3  &\qquad    2.01     &\qquad  0.34     \\
                  &\qquad $ 64$   &\qquad  5.6251e-4  &\qquad    2.00     &\qquad  0.50    \\
                  &\qquad $ 8$    &\qquad  3.2422e-2  &\qquad    *        &\qquad  0.32      \\
                  &\qquad $ 16$   &\qquad  7.9309e-3  &\qquad    2.03     &\qquad  0.32    \\
   $0.35$         &\qquad $ 32$   &\qquad  1.9703e-3  &\qquad    2.01     &\qquad  0.35     \\
                  &\qquad $ 64$   &\qquad  4.9003e-4  &\qquad    2.01     &\qquad  0.55    \\
                  &\qquad $ 8$    &\qquad  2.7408e-2  &\qquad    *        &\qquad  0.30      \\
                  &\qquad $ 16$   &\qquad  6.7136e-3  &\qquad    2.03     &\qquad  0.33    \\
   $0.65$         &\qquad $ 32$   &\qquad  1.6686e-3  &\qquad    2.01     &\qquad  0.39     \\
                  &\qquad $ 64$   &\qquad  4.1499e-4  &\qquad    2.01     &\qquad  0.52    \\
                  &\qquad $ 8$    &\qquad  2.7274e-2  &\qquad    *        &\qquad  0.30      \\
                  &\qquad $ 16$   &\qquad  6.5915e-3  &\qquad    2.05     &\qquad  0.31    \\
   $0.95$         &\qquad $ 32$   &\qquad  1.6340e-3  &\qquad    2.01     &\qquad  0.35     \\
                  &\qquad $ 64$   &\qquad  4.0701e-4  &\qquad    2.01     &\qquad  0.54    \\
     \hline
   \end{tabular}
 \end{table}


  \vskip 1mm
  \textbf{Example 2.} Herein, we present the exact solution by
   \begin{equation}\label{eq5.5}
     \begin{split}
          u(x,t) = \frac{t^{\alpha}}{\Gamma(\alpha+1)} \sin \pi x,
     \end{split}
     \end{equation}
which is more singular than the previous example. Therefore $u_0(x)=0$ and
   \begin{equation}\label{eq5.6}
     \begin{split}
          f(x,t) &= \left( \frac{\pi^2t^{2\alpha}}{\Gamma(2\alpha+1)}  + \frac{t^{\alpha-1}}{\Gamma(\alpha)}   \right)\sin \pi x + \frac{\pi t^{2\alpha}}{2[\Gamma(\alpha+1)]^2} \sin 2\pi x.
     \end{split}
     \end{equation}

\vskip 0.2mm
Herein, for $t>0$, we select $f^{n-\frac{1}{2}}$ by \eqref{eq5.1} in order that the singular behaviour of $f$ does not effect the convergence order (see \cite{McLean}). Besides, we see that
   \begin{equation}\label{eq5.7}
     \begin{split}
          t \|u_{xx}'(\cdot, t)\| + t^2\|u_{xx}''(\cdot, t)\|\leq  \mathcal{M}t^{\alpha},
     \end{split}
     \end{equation}
which illustrates that Theorem \ref{th4.3} holds when $\sigma=1+\alpha$. In Table \ref{tb7}, the results of a uniform mesh report expected temporal convergence orders $k^{1+\alpha}$. And then, with different $\alpha=0.25,0.50,0.75,0.95$, Table \ref{tb8} shows $L^2$ errors, temporal convergence rates and CPU time for $J=512$ and the graded mesh ($\gamma=\frac{2}{\alpha+1}$) when $f^{n-\frac{1}{2}}$ is given by \eqref{eq5.1}, which shows the second-order convergence for time, as predicted.

\vskip 0.2mm
Then for disparate $\alpha=0.05,0.35,0.65,0.95$, Table \ref{tb9} shows the second-order accuracy for space when $N=512$ and a uniform mesh. This point is also demonstrated in Table \ref{tb10} with $\alpha=0.01,0.39,0.69,0.99$, respectively, when $N=256$ and $\gamma=\frac{2}{\alpha+1}$. These all validate the theoretical results.

 \begin{table}
   \centering\small
   \caption{The $L^2$ errors, temporal convergence orders and CPU time (seconds) for $J=512$ and a uniform mesh ($\gamma=1$)
   when $f^{n-1/2}$ is presented by \eqref{eq5.1}.}\label{tb7}
 \begin{tabular}{c c c c c c c}\hline
  $\alpha$  &\qquad $ N $  &\qquad  $E_{CN}$    &\qquad  $Rate^t$   &\qquad  CPU(s)   \\ \hline
                  &\qquad $ 16$    &\qquad  5.4117e-4  &\qquad    *        &\qquad  1.37      \\
                  &\qquad $ 32$    &\qquad  2.0346e-4  &\qquad    1.41     &\qquad  2.57     \\
   $0.25$         &\qquad $ 64$    &\qquad  8.1480e-5  &\qquad    1.32     &\qquad  6.25     \\
                  &\qquad $ 128$   &\qquad  3.4386e-5  &\qquad    1.24     &\qquad  13.04    \\
                  &\qquad $ 16$    &\qquad  2.6227e-3  &\qquad    *        &\qquad  1.44      \\
                  &\qquad $ 32$    &\qquad  8.7778e-4  &\qquad    1.58     &\qquad  2.86     \\
   $0.50$         &\qquad $ 64$    &\qquad  2.9870e-4  &\qquad    1.55     &\qquad  6.44     \\
                  &\qquad $ 128$   &\qquad  1.0389e-4  &\qquad    1.52     &\qquad  13.40    \\
                  &\qquad $ 16$    &\qquad  3.0356e-3  &\qquad    *        &\qquad  1.59      \\
                  &\qquad $ 32$    &\qquad  9.2793e-4  &\qquad    1.71     &\qquad  2.83     \\
   $0.75$         &\qquad $ 64$    &\qquad  2.7537e-4  &\qquad    1.75     &\qquad  5.71     \\
                  &\qquad $ 128$   &\qquad  8.1873e-5  &\qquad    1.75     &\qquad  12.78    \\
                  &\qquad $ 16$    &\qquad  1.7302e-3  &\qquad    *        &\qquad  1.48      \\
                  &\qquad $ 32$    &\qquad  4.5492e-4  &\qquad    1.93     &\qquad  2.60     \\
   $0.95$         &\qquad $ 64$    &\qquad  1.1830e-4  &\qquad    1.94     &\qquad  5.42     \\
                  &\qquad $ 128$   &\qquad  3.0877e-5  &\qquad    1.94     &\qquad  12.53    \\
     \hline
   \end{tabular}
 \end{table}

 \begin{table}
   \centering\small
   \caption{The $L^2$ errors, temporal convergence orders and CPU time (seconds) for $J=512$ and a graded mesh ($\gamma=\frac{2}{\alpha+1}$)
    when $f^{n-1/2}$ is given by \eqref{eq5.1}.}\label{tb8}
 \begin{tabular}{c c c c c c c}\hline
  $\alpha$  &\qquad $ N $  &\qquad  $E_{CN}$    &\qquad  $Rate^t$   &\qquad  CPU(s)   \\ \hline
                  &\qquad $ 8$    &\qquad  7.2543e-4  &\qquad    *        &\qquad  0.74      \\
                  &\qquad $ 16$   &\qquad  1.6705e-4  &\qquad    2.12     &\qquad  1.35     \\
   $0.25$         &\qquad $ 32$   &\qquad  4.0837e-5  &\qquad    2.03     &\qquad  2.50     \\
                  &\qquad $ 64$   &\qquad  1.1431e-5  &\qquad    1.84     &\qquad  5.46    \\
                  &\qquad $ 8$    &\qquad  3.5297e-3  &\qquad    *        &\qquad  0.86      \\
                  &\qquad $ 16$   &\qquad  7.7871e-4  &\qquad    2.18     &\qquad  1.72     \\
   $0.50$         &\qquad $ 32$   &\qquad  1.7826e-4  &\qquad    2.13     &\qquad  2.88     \\
                  &\qquad $ 64$   &\qquad  4.2639e-5  &\qquad    2.06     &\qquad  5.97    \\
                  &\qquad $ 8$    &\qquad  6.0540e-3  &\qquad    *        &\qquad  0.81      \\
                  &\qquad $ 16$   &\qquad  1.6480e-3  &\qquad    1.88     &\qquad  1.53     \\
   $0.75$         &\qquad $ 32$   &\qquad  4.2152e-4  &\qquad    1.97     &\qquad  2.68     \\
                  &\qquad $ 64$   &\qquad  1.0608e-4  &\qquad    1.99     &\qquad  5.25    \\
                  &\qquad $ 8$    &\qquad  5.8041e-3  &\qquad    *        &\qquad  0.85      \\
                  &\qquad $ 16$   &\qquad  1.5193e-3  &\qquad    1.93     &\qquad  1.52     \\
   $0.95$         &\qquad $ 32$   &\qquad  3.8612e-4  &\qquad    1.98     &\qquad  2.67     \\
                  &\qquad $ 64$   &\qquad  9.7184e-5  &\qquad    1.99     &\qquad  5.78    \\
     \hline
   \end{tabular}
 \end{table}

 \begin{table}
   \centering\small
   \caption{The $L^2$ errors, spatial convergence orders and CPU time (seconds) for $N=512$ and a uniform mesh ($\gamma=1$)
    when $f^{n-1/2}$ is given by \eqref{eq5.1}.}\label{tb9}
 \begin{tabular}{c c c c c c c}\hline
  $\alpha$  &\qquad $ J   $  &\qquad  $E_{CN}$    &\qquad  $Rate^x$   &\qquad  CPU(s)   \\ \hline
                  &\qquad $ 8$    &\qquad  9.7674e-3  &\qquad    *        &\qquad  1.12      \\
                  &\qquad $ 16$   &\qquad  2.4272e-3  &\qquad    2.01     &\qquad  1.22     \\
   $0.05$         &\qquad $ 32$   &\qquad  6.0618e-4  &\qquad    2.00     &\qquad  1.30     \\
                  &\qquad $ 64$   &\qquad  1.5179e-4  &\qquad    2.00     &\qquad  1.82    \\
                  &\qquad $ 8$    &\qquad  1.0983e-2  &\qquad    *        &\qquad  1.06      \\
                  &\qquad $ 16$   &\qquad  2.7335e-3  &\qquad    2.01     &\qquad  1.10     \\
   $0.35$         &\qquad $ 32$   &\qquad  6.8844e-4  &\qquad    1.99     &\qquad  1.19     \\
                  &\qquad $ 64$   &\qquad  1.7826e-4  &\qquad    1.95     &\qquad  1.86    \\
                  &\qquad $ 8$    &\qquad  1.1791e-2  &\qquad    *        &\qquad  1.10      \\
                  &\qquad $ 16$   &\qquad  2.9364e-3  &\qquad    2.01     &\qquad  1.20     \\
   $0.65$         &\qquad $ 32$   &\qquad  6.3968e-4  &\qquad    1.99     &\qquad  1.21     \\
                  &\qquad $ 64$   &\qquad  1.9159e-4  &\qquad    1.95     &\qquad  1.90    \\
                  &\qquad $ 8$    &\qquad  1.0631e-2  &\qquad    *        &\qquad  1.07      \\
                  &\qquad $ 16$   &\qquad  2.6397e-3  &\qquad    2.01     &\qquad  1.11     \\
   $0.95$         &\qquad $ 32$   &\qquad  6.5930e-4  &\qquad    2.00     &\qquad  1.21     \\
                  &\qquad $ 64$   &\qquad  1.6499e-4  &\qquad    2.00     &\qquad  1.73    \\
     \hline
   \end{tabular}
 \end{table}

 \begin{table}
   \centering\small
   \caption{The $L^2$ errors, spatial convergence orders and CPU time (seconds) for $N=256$ and a graded mesh ($\gamma=\frac{2}{\alpha+1}$)
    when $f^{n-1/2}$ is given by \eqref{eq5.1}.}\label{tb10}
 \begin{tabular}{c c c c c c c}\hline
  $\alpha$  &\qquad $ J   $  &\qquad  $E_{CN}$    &\qquad  $Rate^x$   &\qquad  CPU(s)   \\ \hline
                  &\qquad $ 4$    &\qquad  3.9177e-2  &\qquad    *        &\qquad  0.27      \\
                  &\qquad $ 8$    &\qquad  9.5538e-3  &\qquad    2.04     &\qquad  0.28     \\
   $0.01$         &\qquad $ 16$   &\qquad  2.3740e-3  &\qquad    2.01     &\qquad  0.30     \\
                  &\qquad $ 32$   &\qquad  5.9263e-4  &\qquad    2.00     &\qquad  0.32    \\
                  &\qquad $ 4$    &\qquad  4.5697e-2  &\qquad    *        &\qquad  0.28      \\
                  &\qquad $ 8$    &\qquad  1.1112e-2  &\qquad    2.04     &\qquad  0.29     \\
   $0.39$         &\qquad $ 16$   &\qquad  2.7597e-3  &\qquad    2.01     &\qquad  0.30     \\
                  &\qquad $ 32$   &\qquad  6.8939e-4  &\qquad    2.00     &\qquad  0.34    \\
                  &\qquad $ 4$    &\qquad  4.8230e-2  &\qquad    *        &\qquad  0.29      \\
                  &\qquad $ 8$    &\qquad  1.1778e-2  &\qquad    2.03     &\qquad  0.29     \\
   $0.69$         &\qquad $ 16$   &\qquad  2.9308e-3  &\qquad    2.01     &\qquad  0.30     \\
                  &\qquad $ 32$   &\qquad  7.3527e-4  &\qquad    1.99     &\qquad  0.33    \\
                  &\qquad $ 4$    &\qquad  4.2250e-2  &\qquad    *        &\qquad  0.29      \\
                  &\qquad $ 8$    &\qquad  1.0310e-2  &\qquad    2.03     &\qquad  0.29     \\
   $0.99$         &\qquad $ 16$   &\qquad  2.5634e-3  &\qquad    2.01     &\qquad  0.30     \\
                  &\qquad $ 32$   &\qquad  6.3922e-4  &\qquad    2.00     &\qquad  0.34    \\
     \hline
   \end{tabular}
 \end{table}


\section{Concluding remarks}

  \vskip 0.2mm
     In this work, an implicit difference scheme has been constructed and analyzed for the nonlinear partial integro-differential equation with a weakly singular kernel. For compensating the singular behavior of the exact solution $u(\cdot,t)$ at $t=t_0$, the non-uniform meshes have been proposed. Then, the discrete energy method was used to derive the stability and convergence of the fully discrete implicit difference scheme. And the existence and uniqueness of numerical solutions were proved based on partly Leray-Schauder theorem. In order to compute proposed implicit difference scheme, an iterative algorithm was employed to obtain approximation solutions. Finally, numerical results have been given to confirm the predicted space-time convergence rates of our approach, which is consistent with the theory.

  \vskip 0.2in
\noindent \textbf{Conflict of interest:} The authors declare that they have no conflict of interest.



\end{document}